\documentclass{article}
\usepackage{amsmath, amsrefs, accents, amsfonts, amscd, amssymb, amsthm,
  verbatim, psfrag, enumerate, dsfont,url, graphicx, multirow, array, indentfirst, 
  tikz, amsthm, wrapfig, mathtools, changepage}
\usepackage[english]{babel}
\usetikzlibrary{matrix, arrows}
\usepackage[font=small,labelfont=bf]{caption}

\usepackage[margin=1.5in]{geometry}

\theoremstyle{plain}
\newtheorem{definition}{Definition}
\newtheorem{theorem}{Theorem}
\newtheorem{lemma}[theorem]{Lemma}
\newtheorem{corollary}[theorem]{Corollary}
\newtheorem{proposition}[theorem]{Proposition}
\newtheorem{exmp}{Example}
\newtheorem{remark}{Remark}

\title{Constrained Motion Spaces of Robotic Arms}
\author{Jack Pierce}
\date{\today}

\begin{document}

\maketitle

\begin{abstract}
	In this paper, we develop the theory of constrained motion spaces of robotic arms.
	We compute their homology groups in two cases:
	when the constraint is a horizontal line
	and when it is a smooth curve whose motion space is a smooth manifold. 
	We show the computation of homology amounts to counting the collinear configurations, 
	reducing a topological problem to a combinatorial problem. 
	Our results rely on Morse theory, along with Walker's and Farber's work on polygonal linkages.
\end{abstract}


\section{Introduction}\label{introduction}

The objective of this paper is to compute the homology of motion spaces of robotic arms
constricted to lines in $\mathbb R^2$. 
A robotic arm is analogous to a sequence of rods attached by hinges, 
where each edge rotates in two dimensions. 
The motion space is the set of all configurations of the robotic arm 
with one edge attached to the origin. 
A constrained motion space is a subspace of the motion space 
where the last vertex is restricted to a subspace of $\mathbb R^2$. 
This definition of motion space better captures the essence of a robotic arm 
than the usual ``configuration" or ``moduli space" (see \cite{walker})
because in real life applications, we often care more about the path
of the ``hand" (last vertex) than the shape of the arm itself. 
We focus on two cases: 
first, when the constraint is a horizontal line; 
second, when the constraint is a smooth curve, with some conditions. 
We hope in the future our results will be extended to any smooth curve in $\mathbb R^2$. 
We will compute the homology using Morse theory. 
Our strategy is adopted from the work of Farber in \cite{farber_polygon}. 

\begin{figure}[h]
	\centering
	\includegraphics[width=0.5\textwidth]{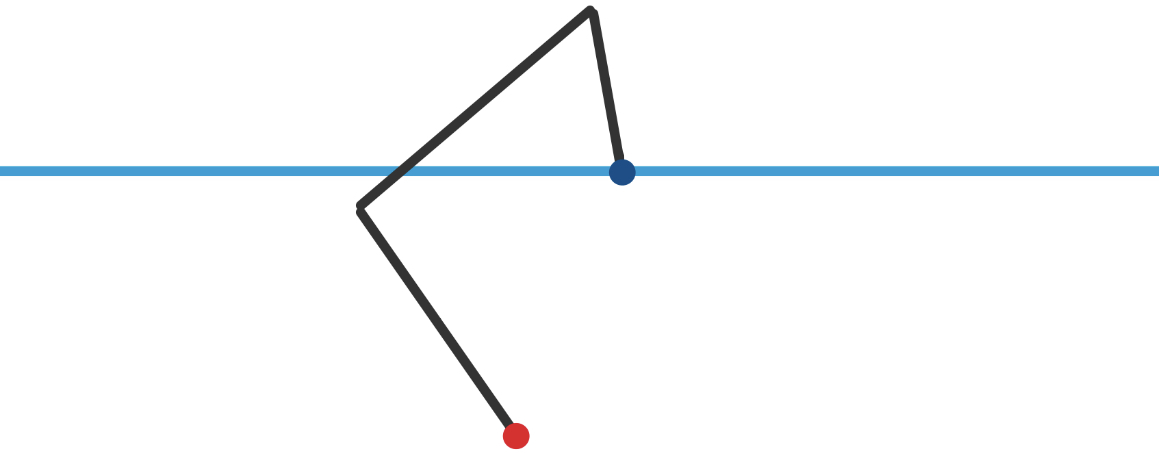}
	\caption{Motion space constricted along a horizontal line}
	\label{fig:y=h_new}
\end{figure}

For the first case, 
let $(\ell_1, \ldots, \ell_k)$ be the lengths of each edge of the robotic arm, 
and let $y = h$ be a horizontal line in $\mathbb R^2$ where $h$ is some real number. 
If we arrange each of the edges to be vertical, i.e. laying on the y-axis, 
then the height of the last vertex is 
\begin{equation}
	\sum_{i \in J} \ell_i - \sum_{i \notin J} \ell_i
\end{equation}
where $J \subseteq \{ 1, \ldots, k \}$ 
is the set of indices of the edges pointed positively (with respect to the y-axis). 
Let $a_j$ be the number of subsets $J \subseteq \{1, \ldots, k \}$ 
such that $|J| = j$ and 
\begin{equation}
	\sum_{i \in J} \ell_i - \sum_{i \notin J} \ell_i \leq -|h|.
\end{equation}
Similarly, let $b_j$ be the number of subsets $J \subseteq \{1, \ldots, k \}$ 
such that $|J| = j$ and 
\begin{equation}
	\sum_{i \in J} \ell_i - \sum_{i \notin J} \ell_i > |h|.
\end{equation}
Furthermore, let $A_h$ be motion space of the robotic arm
constricted to the horizontal line $y = h$. 
With this notation, we state the first result of this paper. 

\begin{theorem}\label{bigtheorem1}
	$H_j(A_h; \mathbb Z)$ is free abelian with rank $a_j + b_{j + 1}$.
\end{theorem}

Thus, we reduce a topological problem to a combinatorial problem
of determining $a_j$ and $b_j$ from $(\ell_i)$. 

For the second case, 
let $\gamma: [0, 1]\to \mathbb R^2$ be an embedding of an interval into the plane 
such that 
\begin{equation}
	|\gamma(0)| = |\gamma(1)| = \sum_{i = 1}^k \ell_i
\end{equation}
and that for each $J \subseteq \{1, \ldots, k \}$, 
$\gamma$ only intersects transversely with the circles
centered at the origin with radius $|r_J|$, where 
\begin{equation}
	r_J = \sum_{i \in J} \ell_i - \sum_{i\notin J} \ell_i 
\end{equation}
which we call ``circles of a critical radius." 
We define $A_\gamma$ to be the motion space of the robotic arm constricted to $\gamma$. 
Then for each subset $J \subseteq \{1, \ldots, k \}$, 
we introduce the concept of the multiplier $\mu_J$, 
defined to be half the number of times the $\gamma$ 
intersects the circle of radius $|r_J|$. 
As such, we redefine 
\begin{equation}
	a_j = \sum \mu_J
\end{equation}
over all $J \subseteq \{1, \ldots, k\}$ where $|J| = j$ 
and $r_J < 0$; 
furthermore, 
\begin{equation}
	b_j = \sum \mu_J
\end{equation}
over all $J \subseteq \{1, \ldots, k \}$ where $|J| = j$
and $r_J > 0$. 
Observe, even if there is a $J$ with $r_J = 0$, 
no intersection with the origin can be transverse. 
Now, we state the second result of this paper. 

\begin{theorem}\label{bigtheorem2}
	$H_j(A_\gamma; \mathbb Z)$ is free abelian with rank $a_j + b_{j + 1}$.
\end{theorem}

\begin{figure}[h]
	\centering
	\includegraphics[width=0.3\textwidth]{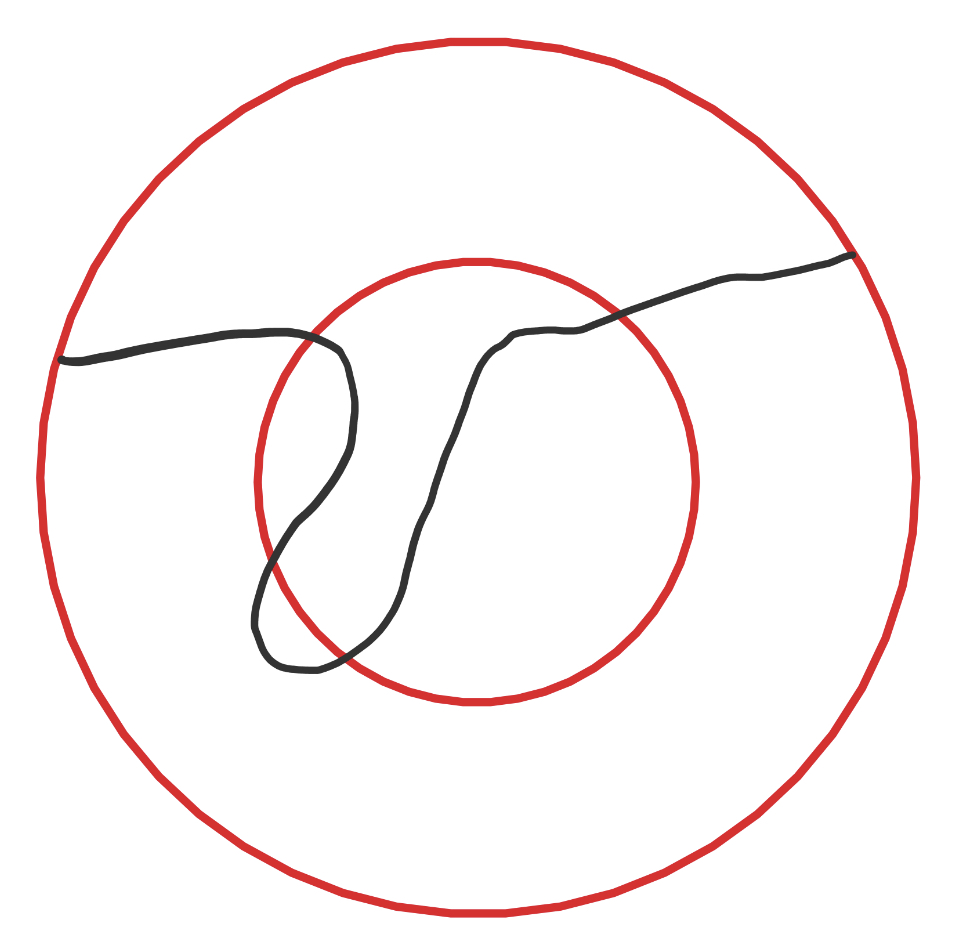}
	\caption{A smooth curve which satisfies the conditions of the second case. The circles of critical radii are in red.}
	\label{fig:smooth_curve_intro}
\end{figure}

We use the same notation of $a_j$ and $b_j$ to show a connection
between the two types of constrictions. 
Namely, a horizontal line without any of the specified tangential intersections
also fits the second case, 
i.e. all $\mu_J = 1$. 
However, the horizontal line also admits certain types of tangential intersections
with circles of critical radii. 
These tangential intersections produce motion spaces which are not smooth manifolds, 
and thus we cannot use the methods used below to compute the homology
of constrictions of the second type. 
Furthermore, there are types of intersections 
which cannot be produced by straight lines, see \underline{Figure~\ref{fig:missing_tangential_intersection}}. 
We have not found a method to compute curves with such intersections. 
This missing piece, along with the similarities between the two theories, 
suggests that there exists a general theory for all smooth curves. 
We leave this work to future papers. 

\begin{figure}[h]
	\centering
	\includegraphics[width=0.3\textwidth]{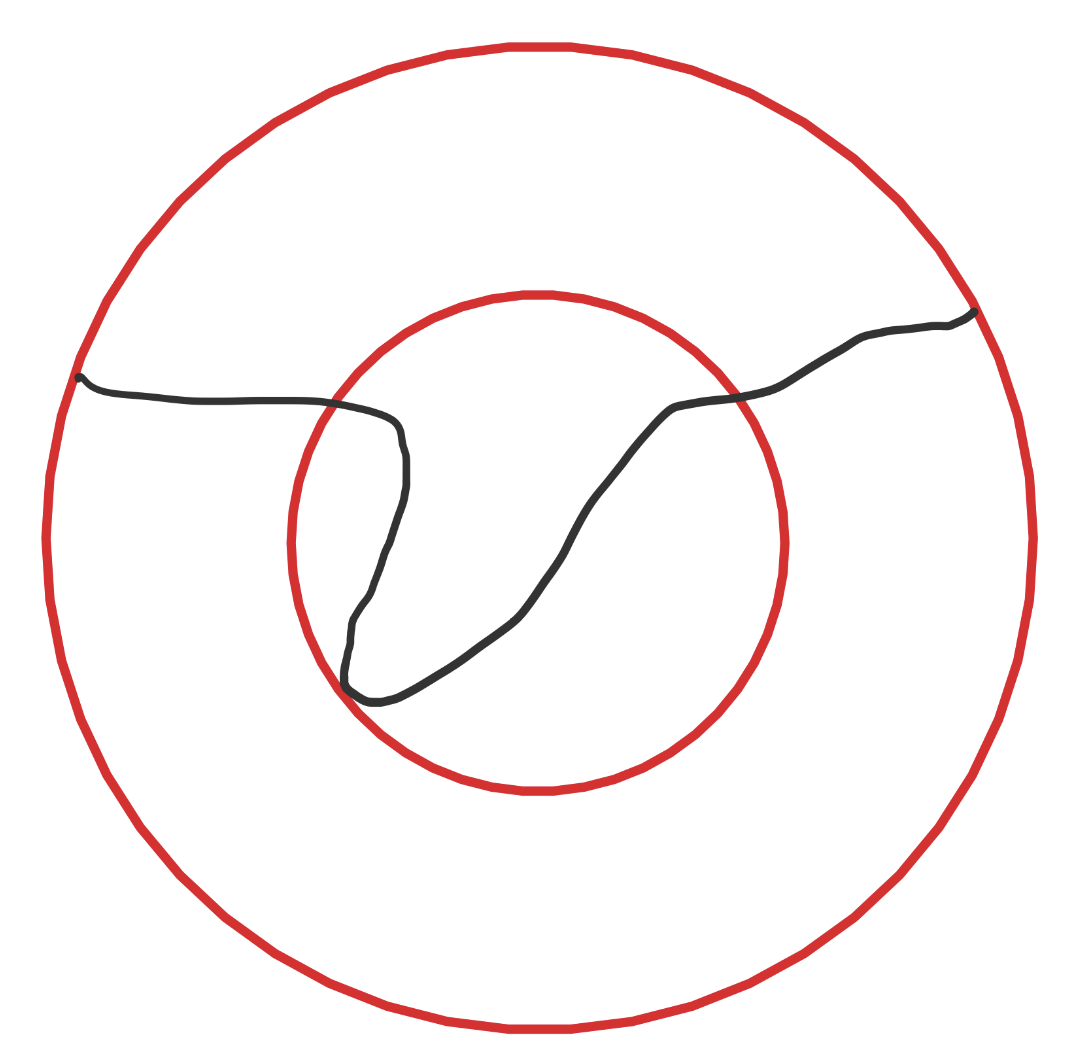}
	\caption{A smooth curve which has a tangential intersection which cannot be replicated with a horizontal line.}
	\label{fig:missing_tangential_intersection}
\end{figure}



\subsection*{Outline}

This paper assumes the reader has familiarity with algebraic topology. 
We try to make all of the concepts as intuitive as possible, 
but some background knowledge is necessary to grasp
the crux of this paper. 
We highly recommend \cite{munkres} and \cite{hatcher}. 

In \underline{Section~\ref{background}}, 
we give much of the background necessary for the rest of the article. 
We first define linkages 
(see \cite{farber_polygon}, \cite{farber_motion}, 
\cite{kourganoff}, \cite{walker}), 
then look at the specific case of a polygon (see \cite{farber_polygon}). 
Afterwards, we define a robotic arm and a motion space. 
before we introduce homology and give some necessary results (see \cite{milnor}, \cite{pino}, \cite{farber_polygon}). 

In \underline{Section~\ref{examples}}, 
we give several examples of motion spaces constricted along lines in $\mathbb R^2$. 
We also apply our theorems to these examples. 

In \underline{Section~\ref{levels}}, 
we prove one last technical result before moving into the proofs. 
We rely heavily on Morse theory. 

Finally, in \underline{Section~\ref{sec:proofs}}, we prove \underline{Theorem~\ref{bigtheorem1}}, 
and \underline{Theorem~\ref{bigtheorem2}}.


\section{Background}\label{background}


\subsection{Linkages and Configuration Spaces}\label{link}

Much of following notation, definitions, and propositions in this section are adapted from \cite{walker}.
There, the reader will find more simple examples to motivate these concepts. 
We rewrite them here for clarity and understandability. 

\begin{definition}
	A \emph{graph} $G = (V, E)$ is a finite set of vertices $V$ 
	along with a set of edges $E \subset V \times V$. 
	The \emph{degree} of a vertex $\nu_1 \in V(G)$ is the number of edges $(\nu_1, \nu_2) \in E(G)$. 
\end{definition}

For our purposes, we identify edges $(\nu_1, \nu_2) \sim (\nu_2, \nu_1)$, 
and we exclude edges of the form $(\nu_1, \nu_1)$. 
We use $V(G)$ to refer to the vertex set of a graph $G$, and $E(G)$ to refer to its edge set. 

\begin{definition}
	A \emph{linkage} $L = (G, \ell)$ is a graph $G$ along with a map $\ell: E(G) \rightarrow \mathbb{R}_+$. 
\end{definition}

In effect, this gives a distance to each edge. 
Moving forward, we refer to the vertices of a linkage as $V(L)$, and $E(L)$ for edges. 
Even though the edges have ``distances,'' we have not actually said what space the graph is in. 
In fact, a graph is an abstract concept and is not ``in'' any space. 
As such, a graph is ``realized" in a space.

\begin{definition}
	A \emph{realization} of a graph $G$ is a map $p: V(G) \rightarrow \mathbb{R}^n$. 
\end{definition}

This fixes each vertex at a particular coordinate in $\mathbb{R}^n$. 
Then each edge $(\nu_1, \nu_2)$ is the line segment connecting $p(\nu_1)$ to $p(\nu_2)$. 
The length of $(\nu_1, \nu_2)$ is $| p(\nu_1) - p(\nu_2) |$, the Euclidean distance in $\mathbb{R}^n$. 
Note this definition permits the lines to overlap. 
We allow for vertices to overlap as well. 

From this definition, we see that $(\mathbb{R}^n)^j$  
is the space of all realizations of a graph $G$ where $j = | V(G) |$. 
However, these realizations do not respect edge lengths. 
Instead, the edge lengths are a function of the realization themselves.

\begin{definition}
	A \emph{configuration} of a linkage $L = (G, \ell)$ with $j$ vertices
	is a realization $p: V(L) \rightarrow \mathbb{R}^n$ 
	such that $|p(\nu_1) - p(\nu_2)| = \ell(\nu_1, \nu_2)$ for all edges $(\nu_1, \nu_2)\in E(L)$. 
\end{definition}

A configuration is one way to arrange the fixed-length edges in $\mathbb{R}^n$ 
so they form the graph $G$. 
We define a notion of closeness between configuration spaces 
if each of the vertices are close to each other in $\mathbb{R}^n$. 
From this, we develop a space of all possible configurations. 

\begin{definition}
	The \emph{free configuration space} of a linkage $L$ in $\mathbb{R}^n$ with $j$ vertices is 
	\[ 
		F_n(L) = \{ p\in (\mathbb{R}^n)^j :
			\ell(\nu_1, \nu_2) = |p(\nu_1) - p(\nu_2)|,   \quad \forall (\nu_1, \nu_2)\in E(L) \} 
	\]
	with the subset topology. 
\end{definition}

The free configuration space of a linkage is all realizations of its graph 
such that the line segments produced by the vertices have length equal to their corresponding edges. 
We call $F_n(L)$ free because it does not have any identification between different configurations
or constraints on what configurations are possible. 
For example, if we rotated a configuration around the origin, 
the two configurations would look similar but be counted as different points in $F_n(L)$. 
There is a similar result from translation and scaling. 
To address this concern, we have the notion of a configuration space.

\begin{definition}
	The \emph{configuration space} of a linkage $L$ in $\mathbb{R}^n$ is
	\[ 
		C_n(L) = F_n(L) / \mathcal E
	\] 
	where $\mathcal E$ is the group of Euclidean transformations in $\mathbb R^n$.
\end{definition}

This identification amounts to ``nailing down'' one of the edges.
Any edge will do, the result is the same for each.  
Equivalently, we could ``nail down" two vertices. 
We use $C(L)$ when it's obvious that $n = 2$. 
Most of our configuration spaces are manifolds. 
Indeed, it is know that for any connected compact manifold, 
there exists some linkage such that the connected components of the configuration space 
are diffeomorphic to the manifold \cite{kourganoff}.


\subsection{Polygonal Linkages}\label{polygon}

One well-studied type of linkage is the polygon. 

\begin{definition}
	A \emph{polygon} $\mathcal{P}_k$ is a linkage $L = (G, \ell)$ with $k$ edges
	where $G$ is connected and every vertex has degree 2. 
\end{definition}

In \cite{walker}, this is referred to as a cyclic linkage. 
For a polygonal linkage, we write $V(G) = \{ \nu_0, \nu_1, \ldots, \nu_{k - 1} \}$, 
and $E(G) = \{ (\nu_0, \nu_1), (\nu_1, \nu_2), \ldots, (\nu_{k - 1}, \nu_0) \}$. 
Further, since $E(G)$ is a finite, ordered set, we treat $\ell$ as an vector of positive reals 
$(\ell_1, \ell_2, \ldots, \ell_k)$ 
where $\ell_i = \ell(\nu_{i - 1}, \nu_i)$ for $i = 1, \ldots, k-1$
and $\ell_k = \ell(\nu_{k - 1}, \nu_0)$ 

We give an example here to illustrate what this means, 
but this theory is developed further below as it is crucial to our understanding of robotic arms. 

\begin{exmp}\label{ex:triangle}
	Consider $C(\mathcal{P}_3)$, the configuration space of a triangle in $\mathbb{R}^2$. 
	Let $\ell = (\ell_1, \ell_2, \ell_3)$ be the side lengths. 
	From geometry, this triangle is unique up to congruency 
	because all the side lengths are fixed. 
	In our configuration space, we account for translation and rotation, but not reflection. 
	We assume without loss of generality that $\ell_1 \geq \ell_2 \geq \ell_3$. 
	Note we prove why we can assume as such in \underline{Proposition~\ref{prop:scaling}}
	and \underline{Proposition~\ref{prop:reordering}}.
	Thus we have 3 cases. 
	The first is $\ell_1 > \ell_2 + \ell_3$. 
	Here $C(\mathcal{P}_3) \cong \emptyset$ because of the triangle inequality. 
	
	Second, we have $\ell_1 = \ell_2 + \ell_3$. 
	The only possible configuration is a line, so $C_2(\mathcal{P}_3) \cong \mathbb{D}^0$, a single point. 
	
	Third, we have $\ell_1 < \ell_2 + \ell_3$. 
	Here is where our polygon looks like a triangle. 
	We fix $\ell_1$ so that it lays on the $x$-axis and $p(\nu_0)$ is on the origin. 
	Then $p(\nu_2)$ is either above or below the $x$-axis, which gives us our only two configurations. 
	This is because configuration spaces take into account rotation and translation but not reflection. 
	Thus in this case, $C_2(\mathcal{P}_3) \cong \mathbb{S}^0$, or two disjoint points. 
\end{exmp}

\begin{figure}[h]
	\centering
	\includegraphics[width=\textwidth]{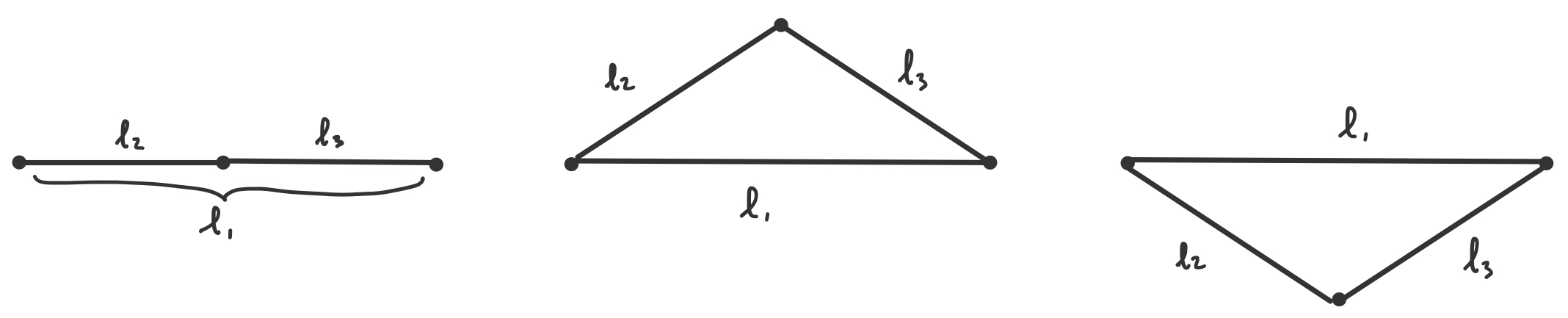}
	\caption{Example Configurations of a Triangle} 
	\label{fig:c2p3}
\end{figure}

Note, the ``flat'' triangle is an example of a collinear configuration. 
We see below these correspond to critical points. 
This is a well known result (see \cite{walker} Proposition 3.3). 
For polygonal linkages, the configuration space is a smooth manifold if the edges lengths admit
no collinear configurations (see \cite{farber_polygon}).


\subsection{Robotic Arms and Their Motion Space}\label{robots_motion}

\begin{definition}
	A \emph{robotic arm} $\mathcal{L}_k$ is a linkage $L = (G, \ell)$ with $k$ edges 
	where $G$ is connected and exactly two vertices have degree 1 while the rest have degree 2. 
\end{definition}

In \cite{walker}, this is referred to as a \emph{chain}. 
We call it a robotic arm because the vertices and edges are analogous to hinges and rods.. 
This alludes to its applications in robotics and motion planning. 
Similar to the polygon, we label the vertices $V(G) = \{ \nu_0, \nu_1, \ldots, \nu_k \}$ 
so that we have edges $E(G) = \{ (\nu_0, \nu_1), (\nu_1, \nu_2), \ldots, (\nu_{k - 1}, \nu_k) \}$
and a length vector $\ell = (\ell_1, \ldots, \ell_k)$ 
where $\ell_i = \ell(\nu_{i - 1}, \nu_i)$. 

\begin{figure}[h]
	\centering
	\includegraphics[width=0.3\textwidth]{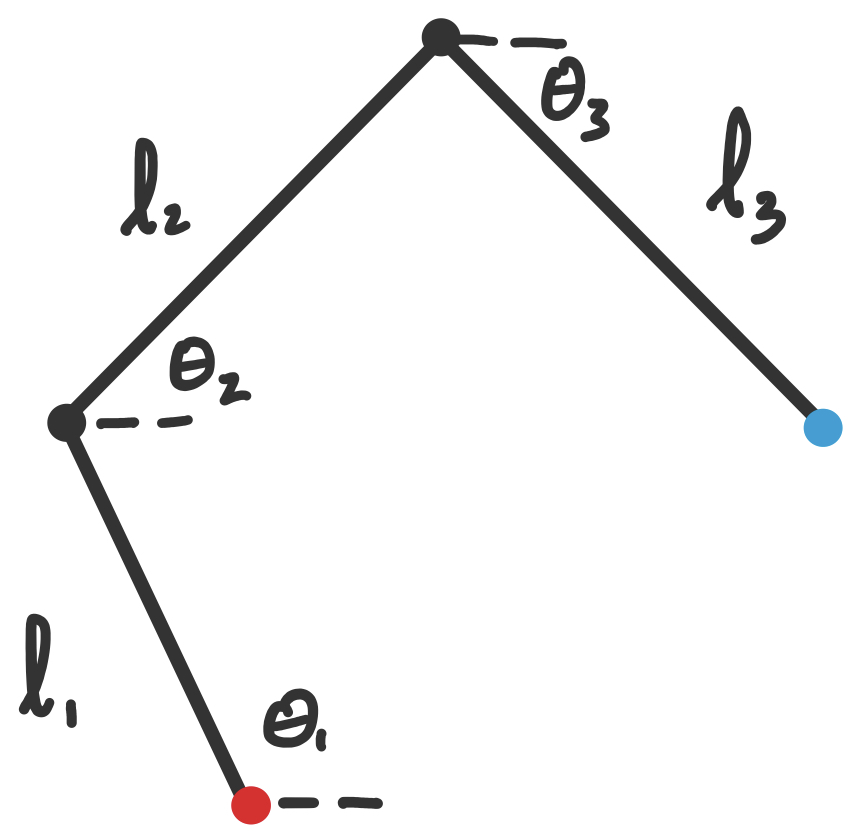}
	\caption{Robot Arm.  Note, the red dot is $\nu_0$ and the blue dot is $\nu_k$.}
	\label{fig:RobotArm}
\end{figure}

Unlike the configuration space defined above, 
we do not want to look at the configurations with one edge ``nailed down"
as that removes a degree of mobility. 
We want all of the edges to be able to freely rotate, 
like an arm attached to a shoulder. 
This means we do not quotient by rotation. 
However, we do still want to quotient by translation. 
This amounts to attaching one end vertex, say $\nu_0$, to the origin. 

\begin{definition}\label{def:motion}
	The \emph{motion space} of a robotic arm $\mathcal{L}_k$ in $\mathbb{R}^n$ is 
	\[ 
		M_n(\mathcal{L}_k) = \{ p\in (\mathbb{R}^n)^k : p(\nu_0) = \textbf{0} \}. 
	\] 
\end{definition}

Note, this means $M_n(\mathcal{L}_k) \times \mathbb{R}^n \cong F_n(\mathcal{L}_k)$ (see \cite{walker}).
In this paper, we only consider the case where $n = 2$, 
so we write $M(\mathcal{L}_k)$. 
Let $\theta_i \in \mathbb{S}^1$ be the angle between edge $(\nu_{i-1}, \nu_1)$ and the $x$-axis. 
Without any constraints on the angles, each rotates freely. 
Thus, the unconstrained motion space of a robotic arm is 
homeomorphic to the $k$-torus 
\begin{equation}\label{eq:torus}
	M(\mathcal{L}_k) \cong \mathbb{T}^k =  \prod^k \mathbb{S}^1,
\end{equation}
the product space of $k$ circles. 

Our ultimate goal is to look at the topology of the motion space of a robotic arm
whose end point moves along any smooth curve. 
In future work, our results will be extended to an arbitrary smooth curve. 
In this paper, we focus on lines in $\mathbb R^2$ and smooth curves with some conditions. 

\begin{definition}\label{def:constrained}
	The \emph{constrained motion space}  
	of a robotic arm $\mathcal{L}_k$ within the subspace $D \subseteq \mathbb R^2$ is
	\[ 
		M(\mathcal{L}_k, D) = \{ p\in M(\mathcal{L}_k) : p(\nu_k) \in D \}. 
	\] 
\end{definition}

The idea for this constrained motion space comes from \cite{kourganoff}. 
Note that if $D = \mathbb R^2$, we have $M(\mathcal{L}_k, D) = M(\mathcal{L}_k)$. 
Below are some examples to illustrate this example more fully. 

\begin{exmp}\label{ex:add_edge}
	Let $z = (x, y) \in \mathbb R^2$ where $z \neq (0, 0)$. 
	Consider $M(\mathcal{L}_k, z)$, the constrained motion space of a robotic arm $\mathcal{L}_k$ 
	with edge lengths $\ell = (\ell_1, \ldots, \ell_k)$ and whose last vertex is fixed at $z$. 
	If we add the edge $(\nu_0, \nu_k)$ with nonzero length $d = \sqrt{x^2 + y^2}$ to our linkage, 
	this produces a polygon $\mathcal{P}_{k + 1}$
	with edge lengths $\ell^\prime = (\ell_1, \ldots, \ell_k, d)$. 
	By fixing $p(\nu_0)$ at the origin and $p(\nu_k)$ at $z$, 
	we have essentially ``nailed down" $(\nu_0, \nu_k)$. 
	Thus, 
	\begin{equation}\label{eq:add_edge}
		M(\mathcal{L}_k, z) = C(\mathcal{P}_{k + 1}).
	\end{equation}
\end{exmp}

\begin{exmp}\label{ExCircleAction}
	Let $z = (0, 0)$. 
	Consider again $M(\mathcal{L}_k, z)$. 
	Here, we cannot add an edge $d = (\nu_0, \nu_k)$ because it would have length 0. 
	However, because $p(\nu_0) = p(\nu_k)$, our robotic arm $\mathcal{L}_k$ is the polygon $\mathcal{P}_k$. 
	In fact, most definitions of a polygon come from a robotic arm whose end vertices are fixed to the same point. 
	It must be noted that the motion space of $\mathcal{L}_k$ 
	is not the same as the configuration space of $\mathcal{P}_k$.
	This is because we want to allow every edge to rotate around the origin in the motion space, 
	but we nail down one of the edges in the configuration space. 
	Nonetheless, if we think of the first edge $(\nu_0, \nu_1)$ as the one nailed down, 
	we induce the motion space of the robotic arm from the configuration space of the polygon by
	rotating this edge around the origin. 
	Thus, 
	\begin{equation}
		M(\mathcal{L}_k, z) = C(\mathcal{P}_{k}) \times \mathbb{S}^1. 
	\end{equation}
\end{exmp}

\begin{exmp}
	Consider the circle $\mathcal{C}_d = \{ (x, y) \in \mathbb R^2 : \sqrt{x^2 + y^2} = d \}$ where $0 < d < 1$. 
	We know the motion space at each point $z \in \mathcal{C}_d$ 
	from \underline{Equation~\ref{eq:add_edge}}, using the same construction of $\mathcal{P}_{k + 1}$. 
	We know $\mathcal{C}_d \cong \mathbb{S}^1$.
	Since the configuration of $M(\mathcal{L}_k, z)$ moves independently from $z$, 
	\begin{equation}
		M(\mathcal{L}_k, \mathcal{C}_d) = C(\mathcal{P}_{k + 1}) \times \mathbb{S}^1. 
	\end{equation}
	A more formal proof exists using sub-linkages and fibered products, 
	the theory of which is developed in \cite{walker}. 
	
\end{exmp}

Now that we understand the relationship between a configuration space and a motion space, 
we prove two propositions. 
The first shows that scaling a linkage does not change its motion space.

\begin{proposition}\label{prop:scaling}
	Suppose $\mathcal L_k$ is a robotic arm with edge lengths $\ell = (\ell_1, \ldots, \ell_k)$. 
	Let $\alpha \in \mathbb{R}_+$. 
	We construct $\mathcal L_k^\prime$ to be another robotic arm with edge lengths $\ell^\prime= (\alpha \ell_1, \ldots, \alpha \ell_k)$.
	Then $M(\mathcal L_k) \cong M(\mathcal L_k^\prime)$. 
\end{proposition}

\begin{proof}
	We show the induced map $\alpha_*: M(\mathcal L_k)\to M(\mathcal L_k^\prime)$ is a homeomorphism. 
	For a configuration $(\theta_1, \ldots, \theta_k)\in M(\mathcal L_k)$, 
	let $\textbf u = (u_1, \ldots, u_k)$ be the vector of complex numbers $u_i = \ell_i \theta_i$
	where we treat $\theta_i \in \mathbb S^1 \subseteq \mathbb C$. 
	We define the induced map such that $\alpha(\textbf u) = (\alpha u_1, \ldots, \alpha u_k)$. 
	Then $\alpha u_i = \alpha (\ell_i \theta_i) = (\alpha \ell_i) \theta_i$ clearly preserves the angle. 
	Thus, $M(\mathcal L_k) \leftrightarrow \mathbb T^k \leftrightarrow M(\mathcal L_k^\prime)$ is a bijection.
	Let $\textbf u^\prime$ be another configuration of $M(\mathcal L_k)$. 
	Then $\alpha(u_i - u_i^\prime) = (\alpha \ell_i)(\theta_i - \theta_i^\prime)$ preserves the angle difference. 
	Therefore, open neighborhoods are mapped to each other, and vice versa. 
\end{proof}

This allows us to assume 
\begin{equation}\label{eq:sum21}
	\sum_{i = 1}^k \ell_i = 1
\end{equation}
for any robotic arm. 
The next proposition shows that reordering of the edge lengths also does not matter.

\begin{proposition}\label{prop:reordering}
	Suppose $\mathcal{L}_k$ is a robotic arm with edge lengths $(\ell_1, \ldots, \ell_k)$. 
	Let $K = \{1, 2, \ldots, k \}$, 
	and let $\sigma: K \rightarrow K$ be a permutation. 
	If we construct $\mathcal{L}_{k}^\prime$ to be a robotic arm with edge lengths $(\ell_1^\prime, \ldots, \ell_k^\prime)$
	where $\ell_i^\prime = \ell_{\sigma(i)}$, 
	then for any given $D \subseteq \mathbb R^2$, 
	we have $M(\mathcal{L}_k, D) \cong M(\mathcal{L}_k^\prime, D)$. 
\end{proposition}

An ``informal argument" is found in \cite{walker}.
We prove it here. 

\begin{proof}
	We define an induced map $\sigma_*: M(\mathcal{L}_k, D) \rightarrow M(\mathcal{L}_k^\prime)$
	and show that this is a self-homeomorphism. 
	
	For a configuration $x\in M(\mathcal{L}_k)$, 
	let $(\theta_1, \ldots, \theta_k) \in \mathbb{T}^k$ be the angles of each edge 
	with respect to the positive x-axis. 
	Then we define $\sigma_*(\theta_1, \ldots, \theta_k) = (\theta_{\sigma(1)}, \ldots, \theta_{\sigma(k)})$. 
	The permutation $\sigma$ is a bijection, which implies $\sigma_*$ is a bijection. 
	
	Let $\varepsilon > 0$. 
	We define
	\begin{gather*} 
		U = \{ (\zeta_1, \ldots, \zeta_k) \in \mathbb{T}^k 
			: | \theta_i - \zeta_i | < \varepsilon, \forall i \} \\
		V = \{ (\xi_1, \ldots, \xi_k) \in \mathbb{T}^k 
			: | \theta_{\sigma(i)} - \xi_i | < \varepsilon, \forall i \}. 
	\end{gather*}
	So $\sigma_*(U) = V$ and ${\sigma_*}^{-1}(V) = U$. 
	We note that sets of this form constitute a basis of $\mathbb{T}_k$. 
	Thus, $\sigma_*$ is a homeomorphism $M(\mathcal{L}_k) \cong M(\mathcal{L}_k^\prime)$. 
	
	Lastly, if we treat the edges as vectors in $\mathbb{R}^2$, 
	we see $p(\sigma_*(\nu_k)) = p(\nu_k)$ because vector addition is commutative. 
	Thus $\sigma_*(M(\mathcal{L}_k, D)) = M(\mathcal{L}_k^\prime, D)$. 
\end{proof}

As such, we may assume $\ell_1 \geq \ell_2 \geq \ldots \geq \ell_k$ for all robotic arms $\mathcal{L}_k$
and all polygonal linkages $\mathcal{P}_k$.


\subsection{Homology}\label{homology}

We now develop the tools necessary compute the homology groups of the constrained motion space. 
We are interested in the Betti numbers $\beta_n$, 
which are equal to the rank of the $n^\text{th}$ integral homology group. 
When there is no torsion (which is the case for all of our motion spaces), 
we understand $\beta_n$ as the number of $n$-dimensional holes in our space. 
Note, $\beta_0$ is the number of connected components.  
For a deeper understanding, see \cite{hatcher}. 
Much of the following notation is adapted from \cite{kourganoff}. 

Because the central proofs of this paper rely on Morse theory, 
we want to have a sense of the critical points of a motion space. 
Recall that a point of a smooth function $f:N\to M$ is critical 
	if at that point the differential has rank less than $\text{dim}M$. 
Suppose we have linkage $\mathcal{L}_k$ with edges $\ell = (\ell_1, \ldots, \ell_k)$ in the plane. 
Let $\theta_i$ denote the angle between each edge and the positive x-axis, 
as in \underline{Figure~\ref{fig:RobotArm}}. 
Then we define 
\begin{equation}\label{eq:last_position}
	s = (s_1, s_2) = \bigg(\sum \ell_i \cos\theta_i, \sum \ell_i \sin\theta_i\bigg) = p(\nu_k) \in \mathbb R^2 
\end{equation}
to be the position of the last vertex of $\mathcal{L}_k$. 
As such, the differential of $s$ is written as the matrix 
\begin{equation}
	Ds = 
	\begin{bmatrix}
		-\ell_1 \sin\theta_1 & -\ell_2 \sin\theta_2 & \cdots & -\ell_k \sin\theta_k \\
		\ell_1 \cos\theta_1 & \ell_2 \cos\theta_2 & \cdots & \ell_k \cos\theta_k
	\end{bmatrix}
\end{equation}
We know that the critical points of $s$ are the configurations $(\theta_1, \ldots, \theta_k)$ 
such that $\text{rk}Ds < 2$. 
This means that all columns are linearly dependent. 
In such a case, for any $i = 2, \ldots, k$ we construct the matrix 
\begin{equation}
	D_i = 
	\begin{bmatrix}
		-\ell_1 \sin\theta_1 & -\ell_i \sin\theta_i \\
		\ell_1 \cos\theta_1 & \ell_i \cos\theta_i
	\end{bmatrix} 
\end{equation}
for which $\det(D_i) = 0$. 
Then by working with the determinant we see
\begin{align*}
	0 & = (-\ell_1 \sin\theta_1)(\ell_i \cos\theta_i) - (-\ell_i \sin\theta_i)(\ell_1 \cos\theta_1)   \\
		& = (\ell_1 \ell_i)(\sin\theta_i \cos\theta_1 - \sin\theta_1 \cos\theta_i)  \\
		& = (\ell_1 \ell_i) \sin(\theta_i - \theta_1).
\end{align*}
We know $\ell_1 \ell_i \neq 0$. 
This means $\theta_i = \theta_1$ or $\theta_i = \theta_1 + \pi$ for all $i = 2, \ldots, k$. 
This conclusion motivates the following definition. 
 
 \begin{definition} 
	If $\theta = (\theta_1, \ldots, \theta_k) \in M(\mathcal{L}_k)$ is a configuration of a robotic arm 
	where $\theta_i = \theta_1$ or $\theta_i = \theta_1 + \pi$ for all $i = 2, \ldots, k$, 
	we call $\theta$ a \emph{collinear} configuration 
\end{definition} 

We call this configuration collinear because all of the edges lay along a single line in $\mathbb{R}^2$. 
From our calculations above, the critical points of $s$ are precisely the collinear configurations. 
However, the codomain of $s$ is $\mathbb{R}^2$. 
When our codomain is $\mathbb{R}$, Morse theory applies, the basics of which we state here. 
We adopt our notation from \cite{pino}. 

We recall the following about Morse theory.
Suppose $f:\mathbb R^n \to \mathbb R$ is a smooth function.
First, the critical point of a $f$ is of Morse type if it is non-degenerate.
Second, if every critical point of $f$ is of Morse type, then $f$ is of Morse type.
Third, the index of a critical point of Morse type is the integer $m$ such that 
with a change of coordinates, the smooth function is locally homeomorphic to the 
polynomial 
\begin{equation}
	\sum_{i = 1}^m -x_i^2 + \sum_{i = m + 1}^n x_i^2
\end{equation}
We show this with the simplest example, the height function on the circle, $\mathbb S^1$.

\begin{exmp}\label{ex:circle_levels}
	Consider $\mathbb S^1 \subset \mathbb R^2$, embedded as the unit circle.
	Let $f:\mathbb S^1 \to \mathbb R$ be a function which projects points on the unit circle
	to the y-axis, so that $f(x, y) = y$ for all $(x, y) \in \mathbb S^1$. 
	If $y = 1$, the pre-image $f^{-1}(y)$ contains only the point $(0, 1)$. 
	Similarly, $f^{-1}(-1)$ contains only the point $(0, -1)$. 
	Otherwise, when $|y| < 1$, the pre-image $f^{-1}(y)$ contains exactly two points $(\pm x, y)$. 
	We visualize this in \underline{Figure~\ref{fig:circle_levels}}. 
	Here, the points $(0, 1)$ and $(0, -1)$ are critical, 
	so the critical values are $1$ and $-1$. 
	Note, near $(0, 1)$ the circle looks similar to the graph of the function $y = -x^2$. 
	Thus, its index is 1. 
	Near $(0, -1)$, the circle looks similar to the graph of $y = x^2$, 
	so its index is 0. 
\end{exmp}

We know for a circle, $\beta_0 = 1$ and $\beta_1 =1$. 
Notice how the Betti numbers match up with the number
of critical points of the same index. 
This is no accident. 
The power of Morse theory is to compute 
the homology of a space by looking at the critical points. 
This particular example is important to the motion space of 
robotic arms with 2 or 3 edges. 

\begin{figure}[t]
	\centering
	\includegraphics[width=0.4\textwidth]{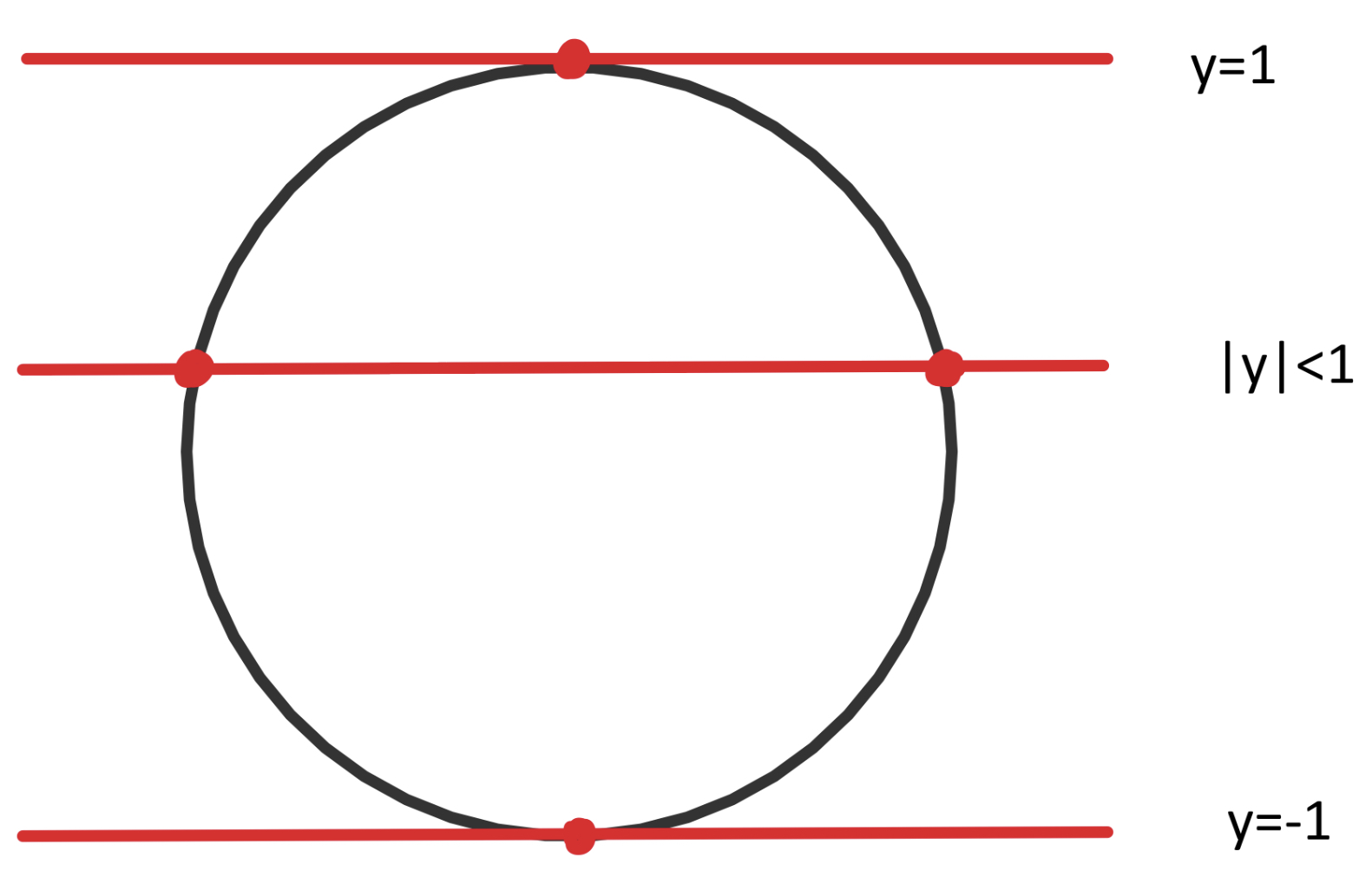}
	\caption{Level Sets of a Circle}
	\label{fig:circle_levels}
\end{figure} 

In order to compute the homology of a space with Morse theory,
we need to look at how the gradient $\nabla f$ acts with respect to the critical points. 
To do so, we bring in two lemmas proved in \cite{farber_polygon}. 

\begin{lemma}\label{farber4}
	Let $M$ be a smooth, compact manifold with boundary. 
	Assume that $M$ is equipped with a Morse function $f: M\to [0, 1]$ 
	and with a smooth involution $\tau: M\rightarrow M$ satisfying the following properties: 
	\begin{enumerate}
		\item $f$ is $\tau$ invariant, i.e. $f(\tau x) = f(x)$ for any $x \in M$; 
		\item The critical points of $f$ coincide with the fixed points of the involution; 
		\item $f^{-1}(1) = \partial M$ and $1 \in [0, 1]$ is a regular value of $f$. 
	\end{enumerate}
	Then each homology group $H_i (M)$ is free abelian of rank 
	equal the number of critical points of $f$ having Morse index $i$. 
\end{lemma}

\begin{lemma}\label{farber5}
	Let $M$ be a smooth, compact, connected manifold with boundary. 
	Suppose that $M$ is equipped with a Morse function $f: M\rightarrow [0, 1]$ 
	and with a smooth involution $\tau: M\rightarrow M$ satisfying 
	the properties of \underline{Lemma~\ref{farber4}}. 
	Assume that for any critical point $p \in M$ of the function $f$, 
	we are given a smooth, closed connected submanifold 
		\[ X_p \subseteq M \]
	with the following properties: 
	\begin{enumerate}
		\item $X_p$ is $\tau$-invariant, i.e. $\tau(X_p) = X_p$;
		\item $p \in X_p$ and for any $x \in X_p - \{ p \}$, one has $f(x) < f(p)$; 
		\item the function $f|_{X_p}$ is Morse and the critical points of the restriction
			$f|_{X_p}$ coincide with the fixed points of $\tau$ lying in $X_p$. 
			In particular, $\dim X_p = \emph{ind}(p)$. 
		\item For any fixed point $q\in X_p$ of $\tau$, the Morse indexes of $f$ and $f|_{X_p}$ 
			at $q$ coincide.
	\end{enumerate}
	Then each submanifold $X_p$ is orientable and the set of homology classes 
	realized by $\{ [X_p] \}_{p \in \text{Fix}(\tau)}$ forms a free basis of the integral homology group $H_*(M)$. 
	Thus the inclusion induces an isomorphism
		\[ \bigoplus_{\emph{ind}(p) = i} H_i (X_p) \to H_i(M) \]
	for any $i$. 
\end{lemma}

These two lemmas operate based on the definition of the Morse-Smale chain complex, 
where we compute the differential based on the gradient paths 
between index points. 
Two gradient paths have opposite orientation and cancel out
if there exists an involution $\tau$ as defined above. 
In such a case, we show that $f$ is a perfect Morse function. 

Note that both of these theorems may be extended to a manifold without boundary vacuously. 
Using these theorems are helpful when we want to consider the inclusion 
of a sublevel set into a manifold.


\section{Examples of Constricted Motion Spaces}\label{examples}

We now give some more examples of constricted motion spaces, 
namely those constricted to motion along a horizontal line in $\mathbb R^2$. 
We give visual representations for these spaces and compute their homology, 
showing they agree with \underline{Theorem~\ref{bigtheorem1}} or \underline{Theorem~\ref{bigtheorem2}}.

For some height $h \in [-1, 1]$, we define 
\begin{equation}\label{eq:Ah}
	A_h = s_2^{-1}(h).
\end{equation}
We analyze $A_h$ by looking at level sets of the restriction $s_1|_{A_h}$
(which we refer to as $s_1$ for clarity). 
Because $A_h$ itself is a level set of $s_2$, 
we are investigating ``level sets of level sets."
For clarity, in this section we only refer to the level sets of $s_1$ as level sets, 
and $A_h$ is the motion space of a robotic arm constricted to a line. 
Lastly for each $J \subseteq \{ 1, \ldots, k \}$, 
we define $u_J = (\theta_1, \ldots, \theta_k)$ to be the configuration 
such that $\theta_i = \frac\pi2$ if $i \in J$ and $\theta_i = -\frac\pi2$ otherwise. 

\begin{wrapfigure}{l}{0.15\textwidth}
	\centering
	\includegraphics[width=0.058\textwidth]{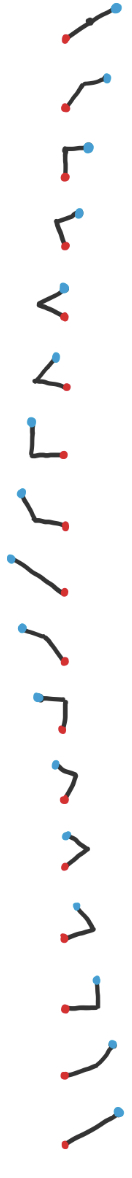}
	\caption{\\ Cycle of \\ configurations \\ of $A_\frac12$ when \\ $\ell = (\frac12, \frac12)$}
	\label{fig:k2half}
\end{wrapfigure} 

Note 
\begin{equation}
	s_2(u_J) = \sum_{i \in J} \ell_i - \sum_{i \notin J} \ell_i. 
\end{equation}

\begin{exmp}\label{ex:pm1}
	Consider $\mathcal L_k$, a robotic arm with length vector $\ell = (\frac1k, \ldots, \frac1k)$. 
	If $h = 1$, then the motion space $A_h \cong \mathbb D^0$ is just a point, 
	namely the configuration $(\frac\pi2, \ldots, \frac\pi2)$ 
	where the linkage is completely ``stretched out." 
	The same happens when $h = -1$ with the configuration $(-\frac\pi2, \ldots, -\frac\pi2)$. 
	
	Since this space is contractible, we know $\beta_0 = 1$ 
	and $\beta_j = 0$ for all $j \geq 1$. 
	Observe $J = \emptyset \subseteq \{ 1, \ldots, k \}$ is the only subset
	such that $s_2(u_J) \leq -1$. 
	Further, there are no subsets $J$ such that $s_2(u_J) > 1$. 
	Then, $a_0 = 1$ while $a_j = 0$ for all $j \geq 1$ and $b_j = 0$ for all $j$. 
	Thus $\beta_j = a_j + b_{j + 1}$ for all $j$,   
	which is the conclusion of \underline{Theorem~\ref{bigtheorem1}}. 
\end{exmp}
	
\begin{exmp}\label{ex:k2half}
	Let $\mathcal L_2$ have edge vector $\ell = (\frac12, \frac12)$. 
	Suppose $h = \frac12$. 
	Consider the level sets $s_1^{-1}(x)$. 
	When $x = \pm \frac{\sqrt{3}}2$,
	the level set is a point, a ``stretched out" collinear configuration
	because $(\frac12)^2 + (\frac{\sqrt{3}}2)^2 = 1$. 
	When $|x| < \frac{\sqrt{3}}2$, 
	there are two possible configurations in the level set. 
	This derives from the two possible configurations outlined in \underline{Example~\ref{ex:triangle}}. 
	Thus, we see that the level sets of $A_h$ are precisely the level sets of a circle, 
	as in \underline{Example~\ref{ex:circle_levels}}. 
	Thus, $A_h \cong \mathbb T^1 = \mathbb S^1$. 
	This loop is displayed in \underline{Figure~\ref{fig:k2half}}, 
	where the red dot is $\nu_0$ fixed at the origin, 
	and the blue dot is $\nu_2$ moving along the line $y = h$. 
	Then, $\beta_0 = 1$ because we have one connected component, 
	and $\beta_1 = 1$ because a circle has one hole in the middle. 
	
	There are 4 possible subsets $J \subseteq \{1, 2 \}$. 
	Thus 
	\begin{equation}\label{eq:k2half}
		s_2(u_J) = 
		\begin{cases}
			1, &\quad J = \{1, 2 \} \\
			0, & \quad J = \{1 \} \text{ or } J = \{ 2 \} \\
			-1, & \quad J = \emptyset
		\end{cases}.
	\end{equation}
	Since $h = \frac12$, there is only one subset such that $s_2(u_J) \leq -\frac12$
	and only one subset such that $s_2(u_J) > \frac12$. 
	Then $a_0 = 1$ since $|\emptyset| = 0$ and $b_2 = 1$ since $|\{1, 2 \}| = 2$, 
	while the other $a_j$ and $b_j$ are 0. 
	Then $\beta_0 = a_0 + b_1$ and $\beta_1 = a_1 + b_2$, 
	which again agrees with \underline{Theorem~\ref{bigtheorem1}}. 
	Note, this is the case for when $0 < |h| < 1$. 
	
	Furthermore, this example agrees with \underline{Theorem~\ref{bigtheorem2}}. 
	Let $\gamma$ be the horizontal line $y = \frac12$. 
	There are two circles of a critical radius, $|r_J| = 1$ and $|r_J| = 0$. 
	Then $\gamma$ only intersects with the larger circle at the points $(\pm \frac{\sqrt{3}}2, \frac12)$. 
	These correspond to the two ``stretched out" configurations. 
	So $\mu_J = 1$ when $J = \emptyset$ or $J = \{ 1, 2\}$, and $\mu_J = 0$ otherwise. 
	This gives $a_0 = 1$, $b_2 = 1$. 
	Otherwise, $a_j, b_j = 0$. 
	Thus, we see $\beta_0 = a_0 + b_1$ and $\beta_1 = a_1 + b_2$. 
\end{exmp}

\begin{exmp}\label{ex:k2zero}	
	Again let $\mathcal L_2$ have edge vector $\ell = (\frac12, \frac12)$.
	Now suppose $h = 0$.  
	The level sets $s_1^{-1}(x)$ where $x\neq 0$ behave similarly to the level sets of a circle. 
	The end level sets $s_1^{-1}(x)$ where $x = \pm 1$ 
	are both points, the two ``stretched out" configurations.
	The level sets $s_1^{-1}(x)$ with $0 < |x| < 1$
	are still just two points. 
	Consider the level set $s_1^{-1}(0)$. 
	Notice that this level set allows infinitely many collinear configurations.
	Thus, $A_h$ is not a smooth manifold.
	However, the level set $s_1^{-1}(0)$ is itself a circle. 
	These are the configurations $(\theta, -\theta)$.  
	The motion space $A_h$
	is homeomorphic to two circles with two points identified. 
	The two points identified are the vertical collinear configurations 
	$(\frac\pi2, -\frac\pi2)$ and $(\frac\pi2, -\frac\pi2)$, 
	visualized in \underline{Figure~\ref{fig:circles2}}. 
	Then $\beta_0 = 1$ and $\beta_1 = 3$
	because this space is homotopic to the wedge of three circles. 
	This is not a manifold because near the configurations 
	$(\frac\pi2, -\frac\pi2)$ and $(\frac\pi2, -\frac\pi2)$, 
	the motion space is homeomorphic to the wedge of two lines, not $\mathbb R^1$. 
	
	Using \underline{Equation~\ref{eq:k2half}}, 
	there are three subsets $J$ such that $s_2(u_J) \leq 0$, 
	and one subset such that $s_2(u_J) > 0$. 
	After differentiating by cardinality of $J$, we see 
	$a_0 = 1$, $a_1 = 2$ and $b_2 = 1$. 
	Then $\beta_0 = a_0 + b_1$ and $\beta_1 = a_1 + b_2$. 
\end{exmp}

\begin{figure}[h]
	\centering
	\includegraphics[width=0.4\textwidth]{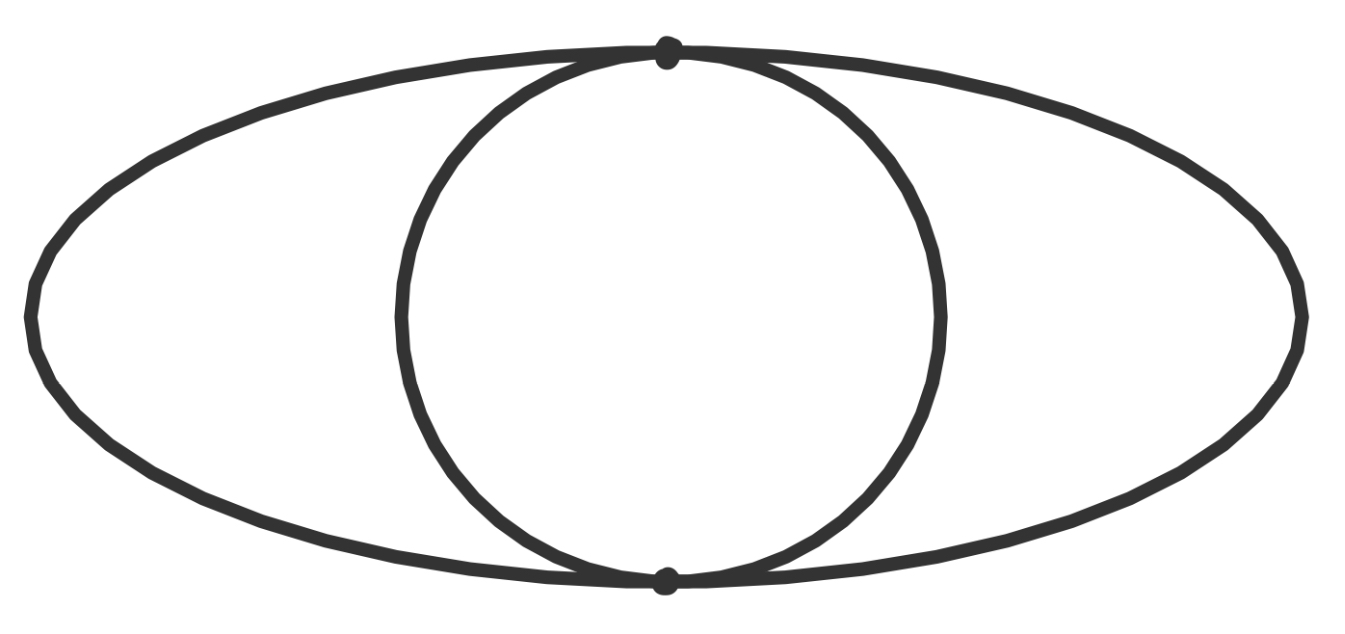}
	\caption{Motion space $A_0$ when $\ell = (\frac12, \frac12)$}
	\label{fig:circles2}
\end{figure} 

There are three points to take away from these last examples. 
First, most of our level sets are of dimension 0, 
which is two less than the number of edges. 
In general, we find that $\dim A_h = k - 1$
because of the Implicit Function Theorem, 
and its level sets have dimension $k - 2$ for the same reason. 

Second, 
the topology of $A_h$ changes when 
the constricted motion space admits more collinear configurations. 
The first result of this paper amounts to computing the Betti numbers of $A_h$
by counting the number of collinear configurations. 

Third, in \underline{Example~\ref{ex:k2zero}}
the horizontal line has a tangential intersection with the origin, 
which produces a level set of dimension 1.
The motion space is not a smooth manifold, 
and its Betti numbers cannot be computed with \underline{Theorem~\ref{bigtheorem2}}.
Next, we consider examples with 3 edges. 

\begin{figure}[h]
	\centering
	\includegraphics[width=0.44\textwidth]{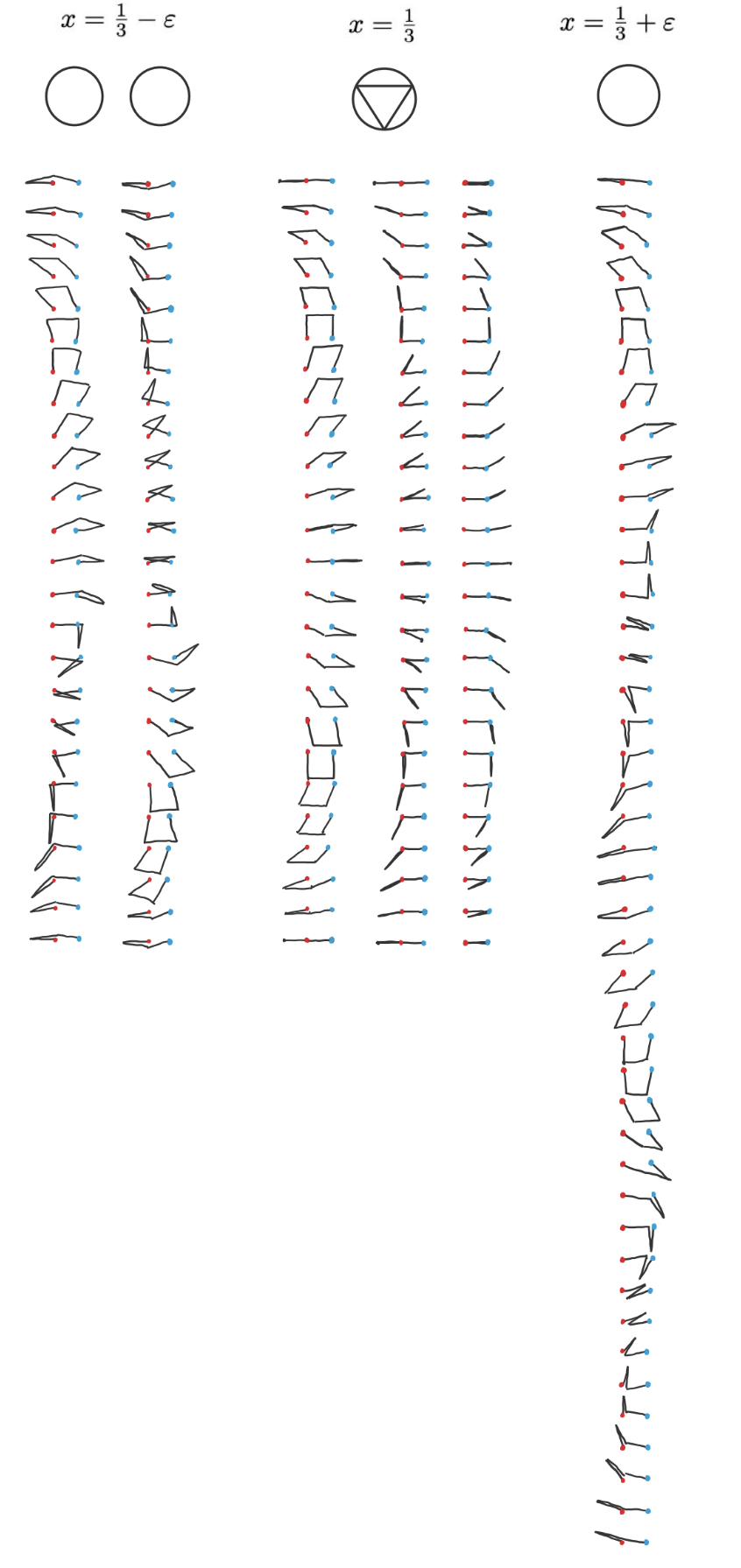}
	\caption{Level Sets $s_1^{-1}(x)$ of $A_0$ when $\ell = (\frac13, \frac13, \frac13)$ 
			\\ and $\varepsilon > 0$ is small}
	\label{fig:ThreeEdges}
\end{figure}

\begin{exmp}\label{ex:k3}
	Let $\mathcal L_3$ be a robotic arm with edge lengths $\ell = (\frac13, \frac13, \frac13)$. 
	Suppose $h = 0$. 
	Note that $s_1^{-1}(x) \cong s_1^{-1}(-x)$ for all $x$
	because each edge reflects over the y-axis. 
	Because we admit 3 collinear configurations at $x = 1$, 
	the level set is two circles with three points identified. 
	We also understand it as one circle with three pairs of points identified. 
	We include the level sets with $x = \frac13 - \varepsilon$ and $x = \frac13 + \varepsilon$ 
	to illustrate how they attach onto the level set at $x = \frac13$
	in \underline{Figure~\ref{fig:ThreeEdges}}. 
	The two notions of how to produce $s_1^{-1}(\frac13)$ 
	come from the two ways the nearby level sets are attached. 
	
	We compute the homology by using $s_1$ as a Morse function. 
	From the differential 
	\begin{equation}
		Ds_1 = 
		\begin{bmatrix}
			-\frac13 \sin\theta_1 & -\frac13 \sin\theta_2 & -\frac13 \sin\theta_3
		\end{bmatrix}
	\end{equation}
	we see that $\emph{rk}Ds_1 < 1$ when $\theta_i \in \{0, \pi\}$ for all $i = 1, 2, 3$. 
	From the restriction to $A_0 = s_2^{-1}(0)$, we know 
	\begin{equation}
		\frac13 \sin\theta_1 + \frac13 \sin\theta_2 + \frac13\sin\theta_3 = 0.
	\end{equation}
	for any $(\theta_1, \theta_2, \theta_3) \in A_h$. 
	Thus, each choice of $\theta_i \in \{ 0, \pi \}$ is a critical point in $A_0$. 
	Further, we see the Hessian 
	\begin{equation}
		H_{s_1} = 
		\begin{bmatrix}
			-\frac13 \cos\theta_1 & 0 & 0 \\
			0 & -\frac13 \cos\theta_2 & 0 \\
			0 & 0 & -\frac13 \cos\theta_3 
		\end{bmatrix}
	\end{equation}
	is non-degenerate for all choices $\theta_i \in \{0, \pi \}$. 
	Therefore, we have $8 = 2^3$ critical points, corresponding to 
	the two choices between 0 and $\pi$ for the three angles $\theta_i$, $i = 1, 2, 3$. 
	
	Looking at the level set $s_1^{-1}(\frac13)$, 
	the critical points $(0, 0, \pi)$, $(0, \pi, 0)$, 
	and $(\pi, 0, 0)$ have index 1. 
	This is because they locally look like level sets of the function $f(x, y) = -x^2 + y^2$. 
	Due to the symmetry $s_1^{-1}(\frac13) \cong s_1^{-1}(-\frac13)$, 
	the critical points $(0, \pi, \pi)$, $(\pi, 0, \pi)$, 
	and $(\pi, \pi, 0)$ also have index 1. 
	The minimum $(\pi, \pi, \pi)$ has index 0, 
	and the maximum $(0, 0, 0)$ has index 2. 
	Observe each of these are fixed points of the involution $\tau^\prime: A_0\to A_0$
	defined by $(\theta_1, \theta_2, \theta_3)\mapsto (-\theta_1, -\theta_2, -\theta_3)$, 
	the reflection over the x-axis. 
	Then, we apply \underline{Lemma~\ref{farber4}}. 
	Thus, $\beta_0 = 1$, $\beta_1 = 6$ and $\beta_2 = 1$. 
	Note this is a smooth manifold. 
	
	Therefore, $A_0$ is a genus 3 surface. 
	
	We now show this agrees with \underline{Theorem~\ref{bigtheorem1}}. 
	There are eight subsets of $\{ 1, 2, 3 \}$, for which we see 
	\begin{equation}
		s_2(u_J) = 
		\begin{cases}
			\; 1, & \quad |J| = 3 \\
			\; \frac13, & \quad |J| = 2 \\
			-\frac13, & \quad |J| = 1 \\
			-1, & \quad |J| = 0
		\end{cases}
	\end{equation}
	Then $a_0 = 1$, $a_1 = 3$, $b_2 = 3$, and $b_3 = 1$, by counting the number of subsets
	of each cardinality, and the other $a_j$ and $b_j$ are 0. 
	Thus $\beta_0 = a_0 + b_1$, $\beta_1 = a_1 + b_2$ and $\beta_2 = a_2 + b_3$. 
	
	We further show this agrees with \underline{Theorem~\ref{bigtheorem2}}. 
	Let $\gamma$ be the horizontal line $y = 0$.
	There are two circles of a critical radius, $|r_J| = 1$ and $|r_J| = \frac13$. 
	Observe $\gamma$ intersects with each circle twice. 
	Thus, $\mu_J = 1$ for all $J \subseteq \{1, 2, 3 \}$. 
	Note there is one subset J for which $|J| = 1$, three for which $|J| = 1$, 
	three more for which $|J| = 2$, and one final subset for which $|J| = 3$. 
	When $|J| \leq 1$, we see $s(u_J) < 0$, 
	and when $|J| \geq 2$, we see $s(u_J) > 0$. 
	Thus $a_0 = 1$, $a_1 = 3$, $b_2 = 3$, and $b_3 = 1$,
	just like with the other method.
\end{exmp}

Observe how, like \underline{Example~\ref{ex:k2half}},
\underline{Example~\ref{ex:k3}} fits both cases 1 and 2:
it is both a horizontal line and has only transverse intersections with the circles of a critical radius. 
The final example only fits case 2. 
Note, this example uses a smooth curve instead of a line.

\begin{figure}[h]
	\centering
	\includegraphics[width=0.3\textwidth]{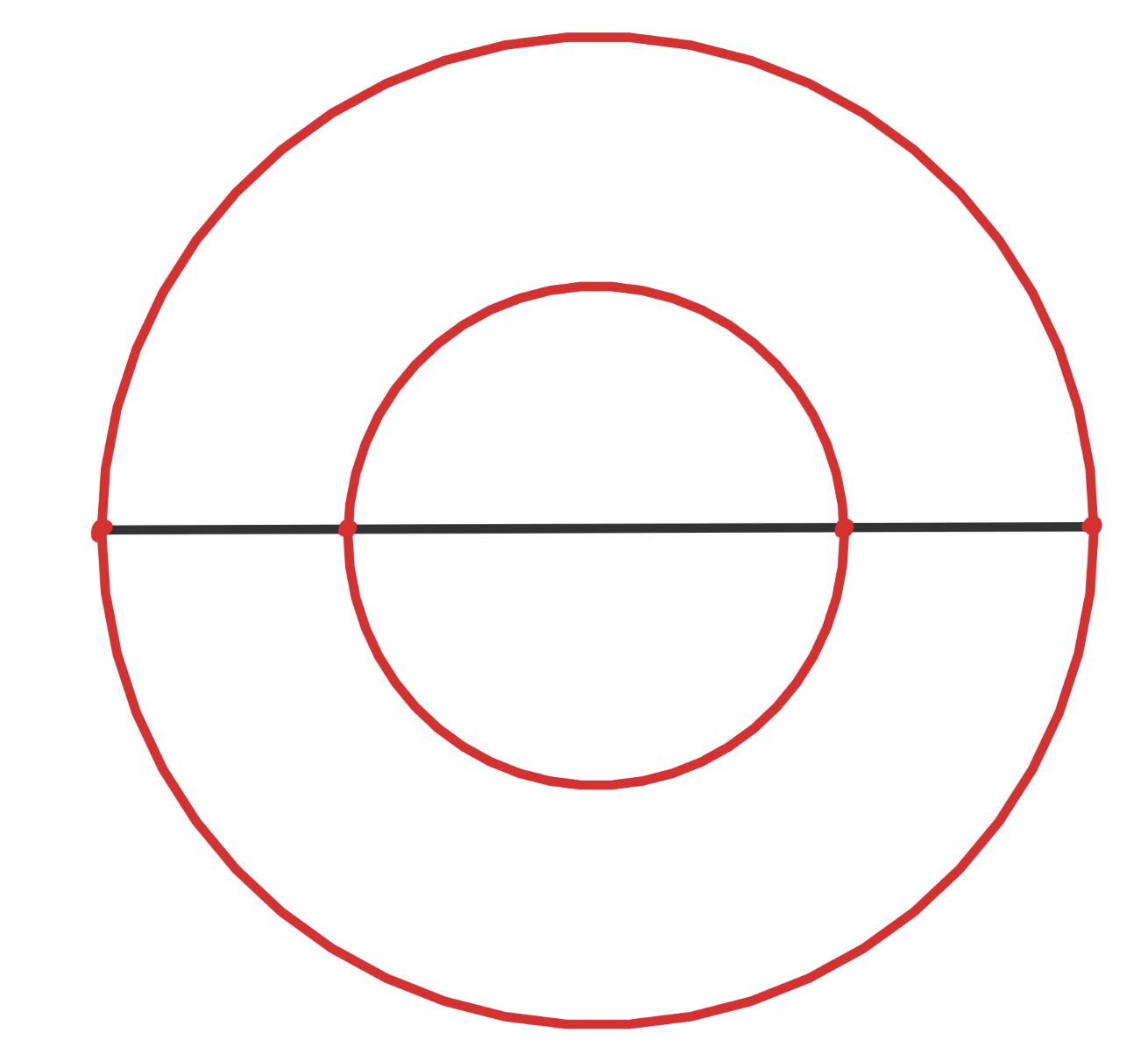}
	\caption{The line $y = 0$ intersects twice with each of the two circles of a critical radius}
	\label{fig:k3_curve}
\end{figure}

\begin{exmp}\label{ex:curve}
	We use the curve from \underline{Figure~\ref{fig:ex_smooth_curve}} as $\gamma$. 
	Suppose $\mathcal L_3$ is a robotic arm with length vector $(\frac13, \frac13, \frac13)$. 
	We first see that there are two intersections with the circle radius 1, 
	and four intersections with the circle of radius $\frac13$. 
	Note 
	\begin{equation}
		r_J = 
		\begin{cases}
			1 &\quad J = \emptyset, \{1, 2, 3 \} \\
			\frac13 &\quad \text{ otherwise }
		\end{cases}
	\end{equation}
	Then we see $\mu_J = 1$ for $J = \emptyset, \{ 1, 2, 3 \}$, 
	while $\mu_J = 2$ otherwise. 
	Further, note that for subsets $J = \emptyset, \{1\}, \{2\}, \{3\}$, 
	we have $s(u_J) < 0$
	Otherwise, $s(u_J) > 0$. 
	Thus we see $a_0 = 1$, $a_1 = 6$ and $a_j = 0$ for all $j \geq 2$; 
	$b_2 = 6$, $b_3 = 1$, and $b_j = 0$ for all $j \neq 2, 3$. 
	Finally, we compute the Betti numbers of $H_*(A_\gamma)$: 
	$\beta_0 = 1$, $\beta_1 = 12$, $\beta_2 = 1$, 
	and $\beta_j = 0$ for all $j \geq 3$. 
\end{exmp}

\begin{figure}[h]
	\centering
	\includegraphics[width=0.3\textwidth]{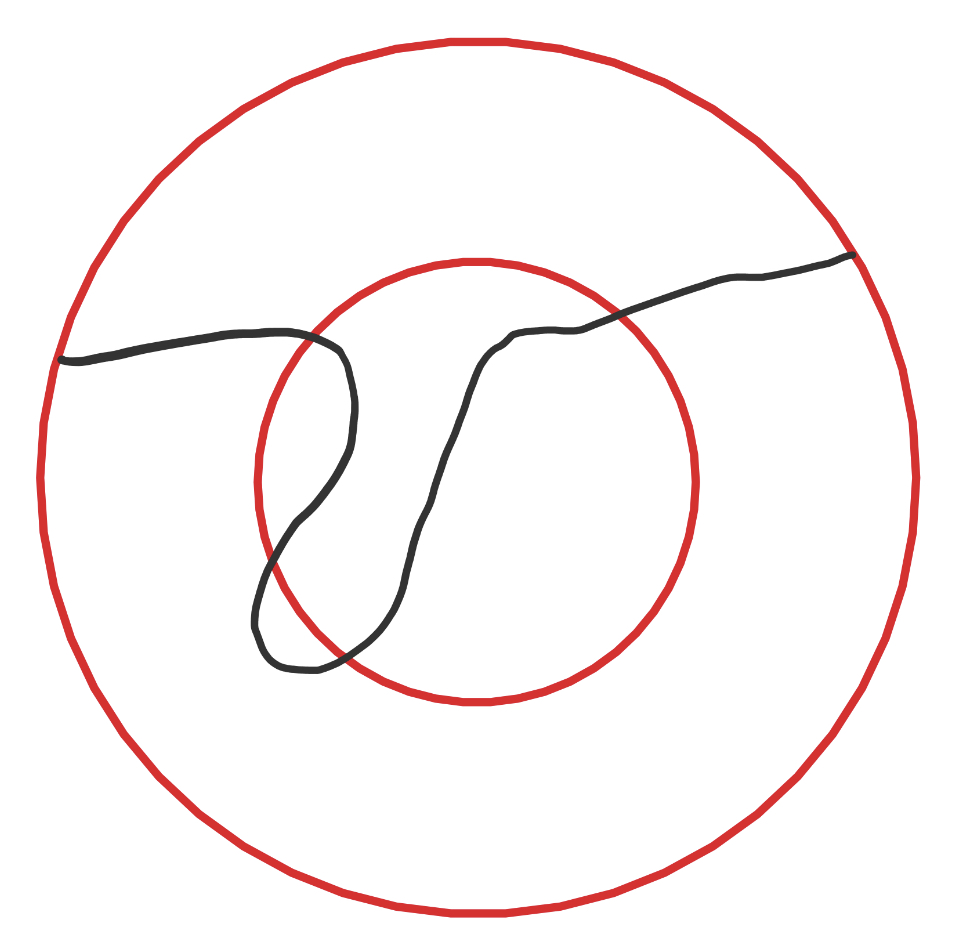}
	\caption{}
	\label{fig:ex_smooth_curve}
\end{figure}


\section{Lemma on Level Sets}\label{levels}

Now that we have stated the necessary background tools, 
we prove a technical theorem. 
One of the critical steps of the proof of \underline{Theorem~\ref{bigtheorem1}}
is to show a deformation retraction onto a level set. 
This is not true in general, but because $M(\mathcal{L}_k)$ is such a nice manifold, 
we show that this is possible. 
We do this using Morse Theory.

We know that a smooth, compact, orientable manifold 
has finitely many critical points
because it must have finitely many nontrivial homology groups, 
each with trivial rank. 
As such, it has finitely many critical values. 
We use this fact to define a deformation retraction onto any level set
in such a manifold. 

Note, this is a well-known result found in \cite{milnor}. 
Though contained in Remarks 3.3 and 3.4, 
we rewrite it here in a more helpful form.

\begin{lemma}\label{def_retract}
	Let $M$ be a smooth, compact, oriented manifold, 
	and let $f: M\to \mathbb R$ be a Morse function. 
	For any $y\in \mathbb{R}$, there exists some $\delta > 0$ 
	such that $f^{-1}[y - \delta, y + \delta]$ deformation retracts
	onto $f^{-1}(y)$. 
\end{lemma}
\begin{proof}
	Let $M$ and $f$ be as stated.
	Consider $y\in \mathbb{R}$. 
	Since there are finitely many critical values, 
	let $\delta > 0$ such that
	no value in $[y - \delta, y) \cup (y, y + \delta]$ is critical. 

	There exists a sufficiently ``nice'' gradient vector field $\nabla f$ such that
	\begin{enumerate}
		\item $\nabla f(p) = 0$ if and only if $p$ is a critical point of $f$, and
		\item $df(\nabla f(p)) > 0$ if $p$ is not a critical point.
	\end{enumerate}
	Note $df$ is a covector field in $M$ 
	which maps the gradient to $\mathbb{R}$ through a dot product. 
	Then we define flow lines $\phi_t(p)$ which ``move'' $p$ along its gradient. 
	Further, 
	\begin{equation}
		\lim_{t\to -\infty}\phi_t(p) = q_0 
			\qquad  \text{ and } \qquad 
		\lim_{t\to\infty}\phi_t(p) = q_1
	\end{equation}
	such that $q_0$ and $q_1$ are critical points where $f(q_0) < f(p) < f(q_1)$
	if $p$ is not critical. 
	In the case where $p$ is critical, $\phi_t(p) = p$ for all $t\in \mathbb{R}$. 
	It is known that $\phi_t$ is smooth (see \cite{pino}). 
	
	If $y$ is not a critical value, then
	\begin{equation}
		f^{-1}(y) \times [y - \delta, y + \delta] \cong f^{-1}[y - \delta, y + \delta] 
	\end{equation}
	is a fundamental result from Morse theory \cite{pino}. 
	We do this by moving each point $p \in f^{-1}(y - \delta)$ along it's flow line 
	to some $p^\prime \in f^{-1}(y + \delta)$. 
	Since there are no critical points, 
	the flow of time $t$ takes $f^{-1}(y)$ to $f^{-1}(y + t)$, 
	and the above homeomorphism is obvious.  
	This is not the interesting case. 
	
	Suppose $y$ is a critical value, 
	and $q_0, \ldots, q_m$ are the critical points such that $f(q_i) = y$. 
	We define a deformation retraction onto $f^{-1}(y)$ using these flow lines. 
	Let 
	\begin{equation}\label{eq:stable}
		S_i = \{ p \in f^{-1}[y - \delta, y] : \lim_{t\to\infty} \phi_t(p) = q_i \}
	\end{equation}
	be the stable submanifold of critical point $q_i$, and let 	
	\begin{equation}\label{eq:unstable}
		T_i = \{ p \in f^{-1}[y, y + \delta] : \lim_{t\to-\infty} \phi_t(p) = q_i \}
	\end{equation}
	be the unstable submanifold. 
	If $q_i$ has index $n_i$, 
	then we know $S_i \cong \mathbb{D}^{n_i}$ 
	such that $S_i \cap f^{-1}(y - \delta) \cong \mathbb{S}^{n_i- 1}$ (see \cite{milnor}). 
	Note $f^{-1}[y - \delta, y]$ is homeomorphic to the mapping cylinder
	of the quotient map $f^{-1}(y - \delta) \to f^{-1}(y)$
	where $S_i \cap f^{-1}(y - \delta)$ is collapsed to a point for each $i = 0, \ldots, m$. 
	This process is understood as sliding each point up its flow path. 
	Mapping cylinders permit deformation retractions onto the codomain, 
	so $f^{-1}[y - \delta, y]$ deformation retracts onto $f^{-1}(y)$. 
	
	Using $T_i$, we similarly construct a deformation retraction from $f^{-1}[y, y + \delta]$
	onto $f^{-1}(y)$. 
	The only difference is that $T_i \cong \mathbb{D}^{\dim M - n_i}$
	so that $T_i \cap f^{-1}(y + \delta) \cong \mathbb{S}^{\dim M - n_i- 1}$. 
	Nevertheless, we are still collapsing each sphere to a point. 
	The composition of these two maps is a deformation retraction
	from $f^{-1}[y - \delta, y + \delta]$ onto $f^{-1}(y)$. 
\end{proof}


\section{Proofs of Main Results}\label{sec:proofs}

\subsection{Proof of \underline{Theorem~\ref{bigtheorem1}}}\label{subsec:proof1}

We finally compute the homology groups  
of a motion space constricted to a horizontal line in $\mathbb{R}^2$. 
Our proof relies on Morse theory. 
We adopt our strategy from \cite{farber_polygon}. 
Farber first used \underline{Lemma~\ref{farber4}} 
and \underline{Lemma~\ref{farber5}} to compute the homology 
of a sublevel set of a larger manifold; 
then showed that the target level set is a deformation retract 
of the compliment of the sublevel set; 
and lastly computed the homology of the target level set 
by splitting a long exact sequence 
and using the intersection to compute the kernel
of the inclusion of the sublevel set into the larger manifold. 
The difference in our proof arises from which sublevel sets we use. 
Namely, we use the overlap of a sublevel set and a supralevel set. 
We show that our target level set is a deformation retract of this intersection. 
Then, we use a long exact Mayer-Vietoris sequence 
to compute the homology of our target level set
by investigating the maps of the inclusion 
of the sub- and supralevel sets. 

\begin{figure}[h]
	\centering
	\includegraphics[scale=0.16]{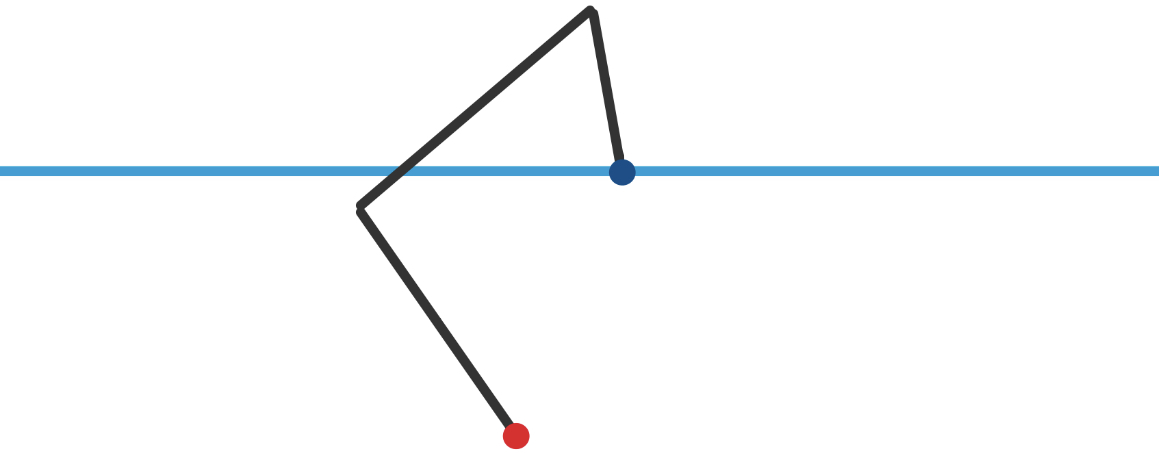}
	\caption{Motion space constricted along a horizontal line}
	\label{fig:y=h}
\end{figure}

This computation ultimately relies on counting the collinear configurations. 
Recall the notation from \underline{Section~\ref{examples}}, where 
$u_J = (\theta_1, \ldots, \theta_k)$ is the configuration such that 
$\theta_i = \frac\pi2$ for all $i \in J$, 
and $\theta_i = -\frac\pi2$ for all $i\notin J$. 
This is a vertical collinear configuration, 
and these are precisely the critical points of $s_2$ from \underline{Equation~\ref{eq:last_position}}. 

We compute the homology groups of the constricted 
motion space along the line $y = h$ for any $h \in \mathbb{R}$. 
We define this space as the level set 
\begin{equation}
	A_h = s_2^{-1}(h) \subseteq M(\mathcal{L}_k).
\end{equation}
Our goal is to compute the homology groups of $A_h$ by counting 
certain vertical collinear configurations
because these are the critical points of $s_2$. 

\begin{proposition}
	The function $s_2$ is of Morse type with critical points $u_J$. 
\end{proposition}
\begin{proof}
	Recall from \underline{Equation ~\ref{eq:last_position}} that we defined $s_2$ as 
	\[ 
		s_2 = \sum_{i = 1}^k \ell_i \sin\theta_i.	
	\] 
	Note, we have 
	\begin{equation}
		Ds_2(\theta_1, \theta_2, \ldots, \theta_k) = 
			\begin{bmatrix} 
				\ell_1 \cos\theta_1 & \ell_2 \cos\theta_2 & \cdots & \ell_k \cos\theta_k 
			\end{bmatrix}.
	\end{equation}
	Since the codomain of $s_2$ is $\mathbb{R}$, 
	the critical points will cause $Ds_2$ to have rank $< 1$. 
	In such a case, every cell must be 0. 
	Since each $\ell_i$ is positive, we know $\cos\theta_i = 0$ for all $i = 1, \ldots, k$. 
	Thus, $(\theta_1, \ldots, \theta_k)$ is a critical point 
	if and only if $\theta_i = \pm\frac\pi2$ for all $i$. 
	
	Then we have the Hessian matrix
	\begin{equation}
		Hs_2(\theta_1, \ldots, \theta_k) = 
			\begin{bmatrix}
				-\ell_1 \sin\theta_1 & 0 & \cdots & 0 \\
				0 & -\ell_2 \sin\theta_2 & \cdots & 0\\
				\vdots & \vdots & \ddots & \vdots \\
				0 & 0 & \cdots & -\ell_k \sin\theta_k
			\end{bmatrix}
	\end{equation}
	which is diagonal. 
	Then the determinant is 
	\begin{align*}
		 |Hs_2(\theta_1, \ldots, \theta_k)| & = (-\ell_1 \sin\theta_1)\cdots(-\ell_k \sin\theta_k) \\
		 	& = (-1)^k(\ell_1 \cdots \ell_k)(\sin\theta_1 \cdots \sin\theta_k).
	\end{align*}
	We know each $\ell_i$ is positive, 
	so the Hessian is zero only when $\sin\theta_i = 0$ for some $i = 1, \ldots, k$. 
	This requires $\theta_i = 0$ or $\pi$. 
	However, at vertical collinear configurations 
	we have $\theta_i = \frac\pi2$ or $-\frac\pi2$ for all $i = 1, \ldots, k$. 
	Therefore, each critical point is of Morse type. 
\end{proof}

In order to work with the supralevel sets, 
we must also prove the same for $-s_2$. 
Let $v_J = (\theta_1, \ldots, \theta_k)$ be the configuration such that 
$\theta_i = -\frac\pi2$ for all $i \in J$, 
and $\theta_i = \frac\pi2$ for all $i\notin J$. 
Thus we see $J = K^c$ if and only if $u_J = v_K$. 
Similarly, $s_2(u_J) = -s_2(v_J)$. 
However, it is helpful to relabel these critical points for $-s_2$
when computing their index. 

\begin{corollary}
	The function $-s_2$ is of Morse type with critical points $v_J$. 
\end{corollary}
\begin{proof}
	First, we see $-Ds_2 = D(-s_2)$. 
	So $(\theta_1, \ldots, \theta_k)$ is a critical point
	if and only if $\theta_i = \pm\frac\pi2$ for all $i = 1, \ldots, k$. 
	Similarly, $-Hs_2 = H(-s_2)$. 
	So each vertical collinear configuration is non-degenerate. 
\end{proof}

Now that we know the Morse Lemma applies
to both $s_2$ and $-s_2$,
we show \underline{Lemma~\ref{farber4}} and 
\underline{Lemma~\ref{farber5}} apply as well.
We say $h$ is a regular value 
if there is no $J \subseteq \{ 1, \ldots, k \}$
such that $h = s_2(u_J)$, 
or equivalently,
no $J$ such that $h = -s_2(v_J)$. 
We define the sublevel sets
\begin{gather}\label{eq:sublevel}
	F_h = s_2^{-1}(-\infty, h] \\
	G_h = (-s_2)^{-1}(-\infty, -h].
\end{gather}
Observe $F_h$ is a supralevel set of $-s_2$, 
and $G_h$ is a supralevel set of $s_2$. 

Before moving to the next result, 
we must define our submanifolds 
which will induce the bases for our homology. 
Let 
\begin{gather}
	S_J = \{ (\theta_1, \ldots, \theta_k) \in M(\mathcal L_k) 
		: \theta_i = -\frac\pi2 \quad \forall i\notin J \} \\
	T_J = \{ (\theta_1, \ldots, \theta_k) \in M(\mathcal L_k) 
		: \theta_i = \frac\pi2 \quad \forall i\notin J \}. 
\end{gather}
Immediately we see $S_J \cong T_J \cong \mathbb T^{|J|}$
because each edge $(\nu_{i - 1}, \nu_i)$ with $i \in J$ 
is rotating freely. 
Thus, $\dim S_J = \dim T_J = |J|$.

\begin{proposition}\label{prop:levels2}
	Suppose $h$ is a regular value.
	Then each homology group $H_j(F_h)$
	is free abelian with rank equal
	the number of sets $J\subseteq \{ 1, \ldots, k \}$
	such that $|J| = j$
	and $s_2(u_J) \leq h$. 
	Further, the submanifolds $\{ [S_J] \}$ 
	form a basis for $H_j(F_h)$. 
\end{proposition}
\begin{proof}
	Define $\tau: F_h\to F_h$ such that
	\begin{equation}\label{eq:tau}
		\tau(\theta_1, \ldots, \theta_k) = (\pi - \theta_1, \ldots, \pi - \theta_k).
	\end{equation}
	Observe this is a reflection over the y-axis.
	We see $\tau$ is $s_2$ invariant by the following
	\begin{align*}
		s_2(\tau(\theta_1, \ldots, \theta_k)) & = s_2(\pi - \theta_1, \ldots, \pi - \theta_k) \\
			& = \sum_{i = 1}^k \ell_i \sin(\pi - \theta_i) \\
			& = \sum_{i = 1}^k \ell_i (\sin\pi \cos\theta_i - \sin\theta_i \cos\pi) \\
			& = \sum_{i = 1}^k \ell_i \sin\theta_i \\
			& = s_2(\theta_1, \ldots, \theta_k).
	\end{align*}
	Note $\pi - \frac\pi2 = \frac\pi2$ and $\pi - (-\frac\pi2) = 3\frac\pi2 = -\frac\pi2$. 
	Thus, the fixed points of $\tau$ are the vertical collinear configurations, which are 
	the critical points of $s_2$. 
	Let $f: [-1, h]\to[0, 1]$ be a linear map 
	such that $f(-1) = 0$ and $f(h) = 1$. 
	Since, $\partial(F_h) = s_2^{-1}(h)$,
	we see $\partial(F_h) = (f\circ s_2)^{-1}(1)$, 
	and $f(h) = 1$ is a regular value. 
	Thus, $\tau$ and $s_2$ satisfy the requirements of \underline{Lemma~\ref{farber4}}. 
	
	To satisfy the requirements of \underline{Lemma~\ref{farber5}}, 
	we first see that $S_J$ is $\tau$ invariant 
	because each $\theta_i = -\frac\pi2$ is mapped to itself.
	Second, $u_J \in S_J $ with 
	\begin{equation}
		s_2(u_J) = \sum_{i \in J} \ell_i - \sum_{i \notin J} \ell_i.  
	\end{equation}
	Then for any $x\in S_J$, we have 
	\begin{equation}
		s_2(x) = \sum_{i \in J} \ell_i \sin\theta_i - \sum_{i \notin J} \ell_i 
	\end{equation}
	which reaches its maximum if and only if $x = u_J$. 
	We already know $\dim S_J = |J|$, 
	and we claim $\text{ind}(u_J) = |J|$ as well. 
	
	We choose the gradient 
	\begin{equation}
		\nabla s_2 = \begin{pmatrix} \ell_1 \cos\theta_1 \\ \vdots \\ \ell_k \cos\theta_k \end{pmatrix} 
	\end{equation}
	so that $\nabla s_2 (x) = 0$ if and only if $x$ is a vertical critical point, 
	and $ds_2(\nabla s_2) > 0$ away from a critical point 
	(see the proof of \underline{Lemma~\ref{def_retract}}). 
	While we will not define the flow lines $\phi_t$ of $\nabla s_2$ in closed form, 
	it suffices to say that for any $\theta_i \neq \pm\frac\pi2$ 
	in a configuration $\theta = (\theta_1, \ldots, \theta_k)$, 
	we have 
	\begin{equation}
		\lim_{t\to\infty} (p_i \circ \phi_t) (\theta) = \frac\pi2 
			\qquad \text{ and } \qquad 
		\lim_{t\to-\infty} (p_i \circ \phi_t) (\theta) = -\frac\pi2. 
	\end{equation}
	where $p_i$ is the projection on to the $i^\text{th}$ factor $\mathbb S^1$. 
	Let $\theta_J = (\theta_i, \ldots, \theta_k)$ be a configuration such that
	$\theta_i = -\frac\pi2$ if $i \notin J$ and $\theta_i \neq -\frac\pi2$ if $i \in J$. 
	So, 
	\begin{equation}
		\lim_{t\to\infty} \phi_t(\theta_J) = u_J. 
	\end{equation}
	For any $\theta_{J^\prime}$ where $J \neq J^\prime$, we have 
	\begin{equation}
		\lim_{t\to\infty} \phi_t(\theta_{J^\prime}) = u_{J^\prime} \neq u_J. 
	\end{equation}
	Then, the stable submanifold (see \underline{Equation~\ref{eq:stable}}) of $u_J$ 
	\begin{equation}
		\{ \theta_J \} = S_J - \bigcup_{J^\prime \subset J} S_{J^\prime}
	\end{equation}
	is homeomorphic to an open disk of dimension $|J|$. 
	Thus, we have $\text{ind}(u_J) = \dim S_J = |J|$. 
	This holds for any restriction $s_2|_{S_J}$ because $S_{J^\prime} \subseteq S_J$ 
	for any $J^\prime \subseteq J$. 
\end{proof}

\begin{remark}\label{remark1}
	Therefore, the inclusions of $S_J$ induce isomorphisms
	\begin{equation}
		\bigoplus_{|J| = j} H_j (S_J) \to H_j (M(\mathcal{L}_k)) 
	\end{equation}
	for any $j = 0, \ldots, k$. 
	Note that there are ${n \choose j}$ subsets $|J| = j$. 
	Then $H_j(S_J) \approx \mathbb{Z}$ for any $|J| = j$. 
	Observe we have proved $H_j(\mathcal{L}_k)$ is free abelian with rank ${ k \choose j}$, 
	though this is already known since it is homeomorphic to a k-torus. 
	However, the strength of \underline{Lemma~\ref{farber4}} 
	and \underline{Lemma~\ref{farber5}}
	lies in the ability to construct a similar isomorphism for sublevel sets of $s_2$. 
	For any regular value $h \in [-1, 1]$, we define 
	\begin{equation}\label{eq:Uj}
		\mathcal{U}_j^h = \Big\{ J \subseteq \{1, \ldots, k \} : s_2(u_J) \leq h, \quad  |J|=j\Big\} 
	\end{equation}
	which induce the isomorphisms 
	\begin{equation}
		\bigoplus_{J \in \mathcal{U}_j^h} H_j(S_J) \to H_j(F_h) 
	\end{equation}
	showing the integral homology groups of $F_h = s_2^{-1}[-1, h]$ are free abelian with rank $|\mathcal{U}_j^h|$. 
	We construct the injective map 
	\begin{equation}\label{eq:injectiveU}
		H_j(F_h) \to H_j (M(\mathcal{L}_k)) 
	\end{equation}
	which maps $[S_J]\to [S_J]$. 
	This is all the information we need on the sublevel sets of $s_2$. 
\end{remark}

\begin{corollary}\label{prop:level-s2}
	Suppose $h$ is a regular value.
	Then each homology group $H_j(G_h)$
	is free abelian with rank equal
	the number of sets $J\subseteq \{ 1, \ldots, k \}$
	such that $|J| = j$
	and $-s_2(v_J) \leq -h$. 
	Further, the submanifolds $\{ [T_J] \}$ 
	form a basis for $H_j(G_h)$. 
\end{corollary}

This follows directly from \underline{Proposition~\ref{prop:levels2}} 
and does not require its own proof. 

\begin{remark}\label{remark2}
	The inclusions induce isomorphisms 
	\begin{equation}
		\bigoplus_{|J| = j} H_j (T_J) \to H_j (M(\mathcal{L}_k)) 
	\end{equation}
	for any $j = 0, \ldots, k$. 
	For any regular value $h\in [-1, 1]$, we define 
	\begin{equation}\label{eq:Vj}
		\mathcal{V}_j^h = \Big\{ J \subseteq \{1, \ldots, k \} : -s_2(v_J) \leq -h, \quad  |J|=j\Big\} 
	\end{equation}
	which induce the isomorphisms 
	\begin{equation}
		\bigoplus_{J \in \mathcal{V}_j^h} H_j(T_J) \to H_j(G_h) 
	\end{equation}
	showing the integral homology groups of $G_h = (-s_2)^{-1}[-1, -h]$ are free abelian with rank $|\mathcal{V}_j^h|$. 
	We construct an injective map (similar to \underline{Equation~\ref{eq:injectiveU}})
	\begin{equation}\label{eq:injectiveV}
		H_j(G_h) \to H_j (M(\mathcal{L}_k)) 
	\end{equation}
	which maps $[T_J]\to [T_J]$. 
	This is all the information we need on the supralevel sets of $s_2$.
\end{remark}

\begin{remark}
	Note $M(\mathcal{L}_k) \cong \mathbb T^k$ is a smooth, compact, oriented manifold.
	Because $s_2$ is a Morse function,
	there exists a $\delta > 0$ such that $s_2^{-1}[h - \delta, h + \delta]$ 
	deformation retracts onto $s_2^{-1}(h)$ for all $h\in \mathbb{R}$
	by \underline{Lemma~\ref{def_retract}}. 
	In this case, $\delta$ small enough so that there are no
	critical values in $[h - \delta, h) \cup (h, h + \delta]$. 
	Thus, we see $F_{h + \delta} \cap G_{h - \delta}$ 
	deformation retracts onto $A_h$. 

	Let $F = F_{h + \delta}$ and $G = G_{h - \delta}$. 
	Then, we construct the following Mayer-Vietoris sequence:

	\begin{tikzpicture}[descr/.style={fill=white,inner sep=1.5pt}]
     	   \matrix (m) [
	            matrix of math nodes,
	            row sep= 1em,
	            column sep=3em,
	            text height=1.5ex, text depth=0.25ex
	        ]
	        { 0	& H_k(A_h) 			
	        			& H_k(G) \oplus H_k(F) 				
				& H_k(M(\mathcal{L}_k)) 	& 	\\
	             	& H_{k-1}(A_h) 		
				& H_{k-1}(G) \oplus H_{k-1}(F) 		
				& H_{k-1}(M(\mathcal{L}_k)) 	& 	\\
	             	& \mbox{}         &                & \mbox{}  			&  	\\
	             	& H_0(A_h) 			
				& H_0(G) \oplus H_0(F) 				
				& H_0(M(\mathcal{L}_k)) 	& 0	\\
	        };
	
	        \path[overlay,->, font=\scriptsize,>=latex]
	        (m-1-1) edge (m-1-2)
	        (m-1-2) edge (m-1-3)
	        (m-1-3) edge node[descr,yshift=0.3ex] {$\psi^k$} (m-1-4)
	        (m-1-4) edge[out=355,in=175] (m-2-2)
	        (m-2-2) edge (m-2-3)
	        (m-2-3) edge node[descr,yshift=0.3ex] {$\psi^{k - 1}$} (m-2-4)
	        (m-2-4) edge[out=355,in=175,dashed] (m-4-2)
	        (m-4-2) edge (m-4-3)
	        (m-4-3) edge node[descr,yshift=0.3ex] {$\psi^0$} (m-4-4)
	        (m-4-4) edge (m-4-5);
	\end{tikzpicture}
	
	from which we produce the short exact sequences
	\begin{equation}
		0 \to \text{coker } \psi_{j + 1} \to H_j(A_h) \to \ker \psi_j \to 0. 
	\end{equation}
	We know $H_*(G) \oplus H_*(F)$ is free, 
	so $\ker \psi_*$ is free. 
	Then, this short exact sequence splits 
	\begin{equation}\label{eq:split}
		H_j(A_h) \approx \text{coker } \psi_{j + 1} \oplus \ker \psi_j. 
	\end{equation}
	As such, we need only investigate $\psi_*$. 
\end{remark}

Using the notation from \underline{Equation~\ref{eq:Uj}}
and \underline{Equation~\ref{eq:Vj}}, 
let $\mathcal U_j = \mathcal U_j^{h + \delta}$
and let $\mathcal V_j = \mathcal V_j^{h - \delta}$. 

\begin{proposition}\label{prop:ker}
	The kernel $\ker \psi_j = \Big\langle ([T_K], -[S_J]) : J = K \Big\rangle$
	is free abelian with rank
	\begin{equation}\label{eq:kernel_rank}
		|\mathcal V_j \cap \mathcal U_j|.
	\end{equation}
\end{proposition}
\begin{proof}
	Consider the family of functions $w_t: M(\mathcal L_k)\to M(\mathcal L_k)$,
	$t\in [0, 1]$, where
	\begin{equation}
		w_t(\theta_1, \ldots, \theta_k) = (\theta_1 + \pi t, \ldots, \theta_k + \pi t).
	\end{equation}
	Then for any $J \subseteq \{ 1, \ldots, k \}$, 
	we have $w_0(S_J) = S_J$
	and $w_1(S_J) = T_J$. 
	Thus $[S_J] = [T_J]$, 
	and $w_t$ induces a change of basis map
	on $H_*(M(\mathcal L_k))$. 
	Because both $\{[S_J]\}$ and $\{ [T_J] \}$
	are free bases, 
	they can have no other relationships. 
	
	Given $|J| = |K| = j$, 
	we know $\psi_j([T_K], [S_J]) = [T_K] + [S_J]$
	from the construction of the Mayer-Vietoris sequence. 
	Then, $\psi_j([T_J], -[S_J]) = 0$ 
	when $T_J\subseteq G$ 
	and $S_J\subseteq F$. 
	This only happens when $J \in \mathcal V_j \cap U_j$. 
\end{proof}

With this result, we compute the cokernel. 

\begin{corollary}\label{cor:coker}
	The cokernel $\emph{coker}\psi_j = H_j( \mathbb T^k) / \text{im} \psi_j $
	is free abelian with rank
	\begin{equation}\label{eq:cokernel_rank}
		{k \choose j} - |\mathcal V_j \cup \mathcal U_j| = |(\mathcal V_j \cup \mathcal U_j)^c|. 
	\end{equation}
\end{corollary}
\begin{proof}
	We know $H_j(M(\mathcal L_k))$ is free abelian with rank ${n \choose j}$. 
	From \underline{Proposition~\ref{prop:levels2}} and \underline{Proposition~\ref{prop:level-s2}},
	both $\{ [S_J] : j = |J| \}$ and $\{ [T_J] : j = |J| \}$ are bases. 
	From \underline{Remark~\ref{remark1}} and \underline{Remark~\ref{remark2}}
	we know that 
	\begin{equation}
		\text{im}\psi_j = \Big\langle [S_J], [T_K] \Big| J\in \mathcal U_j, K\in \mathcal V_j \Big\rangle
	\end{equation}
	the subgroup generated by the bases elements $S_J \subseteq F$ and $T_K \subseteq G$. 
	Because of \underline{Proposition~\ref{prop:ker}}, 
	we relate the basis elements $[S_J] = [T_J]$. 
	Thus, 
	\begin{equation}
		\text{im}\psi_j = \Big\langle [S_J] \Big| J\in \mathcal U_j \cup \mathcal V_j \Big\rangle.
	\end{equation}
	Therefore, the cokernel
	\begin{equation}
		\text{coker}\psi_j =  \Big\langle [S_J] \Big| J\in (\mathcal U_j \cup \mathcal V_j)^c \Big\rangle
	\end{equation}
	is generated by the basis elements in neither $F$ nor $G$, 
	and must have rank $|(\mathcal U_j \cup \mathcal V_j)^c|$. 
\end{proof}

These are all the pre-requisites for the last step of the proof. 
First let $a_j$ be the number of vertical collinear configurations $u_J$ 
such that $|J| = j$ and $s_2(u_J) \leq -|h|$.
Second, let $b_j$ be the number of vertical collinear configurations $u_J$
such that $|J| = j$ and $s_2(u_J) > |h|$. 

\begin{remark}\label{rmk:horizontal}
	Using \underline{Equation~\ref{eq:split}}, 
	along with \underline{Proposition~\ref{prop:ker}}
	and \underline{Corollary~\ref{cor:coker}}, 
	We see that $H_j(A_h)$ is free abelian with rank 
	$|\mathcal U_j \cap \mathcal V_j| + |(\mathcal U_{j + 1} \cup \mathcal V_{j + 1})^c|$
	for each $j = 1, \ldots, k$. 
	Observe $\mathcal U_j^{h + \delta} = \mathcal U_j^h$ because there are no critical values
	in the range $(h, h + \delta]$. 
	Similarly $\mathcal V_j^{h - \delta} = \mathcal V_j^h$. 
	Thus $J \in \mathcal V_j \cap \mathcal U_j$ if and only if $-s_2(v_J) \leq -h$ and $s_2(u_J) \leq h$. 
	This is equivalent to requiring $s_2(u_J) \leq -|h|$
	which is the definition of $a_j$. 
	Therefore,
	\begin{equation}
		a_j = |\mathcal U_j \cap \mathcal V_j|.
	\end{equation}
	
	Similarly, we see $J \notin \mathcal V_j \cup \mathcal U_j$ if and only if $-s_2(v_J) > -h$ and $s_2(u_J) > h$. 
	This is equivalent to requiring $s_2(u_J) > |h|$,  
	which is the definition of $b_j$.
	Thus, 
	\begin{equation}
		b_j = |(\mathcal U_j \cup \mathcal V_j)^c|.
	\end{equation}
	
	Therefore, for any $h\in \mathbb{R}$, each homology group $H_j(A_h)$ is free abelian of rank
	\begin{equation}
		 \beta_j = b_{j + 1} + a_j
	\end{equation}
	proving \underline{Theorem~\ref{bigtheorem1}}. 
\end{remark}


\subsection{Proof of \underline{Theorem~\ref{bigtheorem2}}}\label{sec:proof2}

We now consider motion spaces restricted to a smooth curve. 
We repeat much of the notation defined in \underline{Section~\ref{introduction}} for clarity. 
We begin with the robotic arm $\mathcal L_k$ with edge vector $(\ell_1, \ldots, \ell_k)$. 
Similar to above, we must pay attention to the collinear configurations
admitted by the constriction. 
Recall, we define
\begin{equation} 
	r_J = \sum_{i\in J} \ell_i - \sum_{i\notin J} \ell_i,
	\qquad J\subseteq \{1, \ldots, k \}
\end{equation}
and will refer to circles of a critical radius, meaning a circle centered on the origin 
such that the radius is equal to $|r_J|$, for some $J \subseteq \{1, \ldots, k \}$. 
We let $\gamma: [0, 1]\to \mathbb R^2$ be an embedding of the interval in the plane
such that $| \gamma(0) | = | \gamma(1) | = 1$ 
and all intersections with circles of a critical radius are transverse. 
Next, we need the concept of a multiplier $\mu_J$ for each $J\subseteq \{1, \ldots, k \}$,
which we define as half the number of times $\gamma$ 
intersects the circle with critical radius $r_J$. 
For each $j = 0, 1, \ldots$, we redefine 
\begin{equation}
	a_j = \sum \mu_J
\end{equation}
over all $J \subseteq \{1, \ldots, k \}$ such that $|J| = j$ and $r_J < 0$. 
Such a $J$ is ``short."
Similarly, for each $j = 0, 1, \ldots$, we redefine
\begin{equation}
	b_J = \sum \mu_J
\end{equation}
over all $J \subseteq \{1, \ldots, k \}$ such that $|J| = j$ and $r_J > 0$. 
Such a $J$ is ``long."
Lastly, we define $A_\gamma = M(\mathcal L_k, \gamma)$. 
Our goal is to show $H_j(A_\gamma)$ is free abelian with rank $a_j + b_{j + 1}$ for all $j \geq 0$. 

Our first inclination of why this is true comes from our work with straight lines. 
Since $\gamma$ is smooth, locally it is homeomorphic to a straight line. 
With these local homeomorphisms, 
we construct local homeomorphisms between $A_\gamma$ 
and some motion space of a line, 
including when $\gamma$ intersects with a circle of critical radius. 
Away from these circles with a critical radius, 
the constricted motion space should be a smooth manifold with boundary. 
We calculate the homology states above by 
constructing a Morse function and then applying \underline{Lemma~\ref{farber4}}. 

\begin{remark}\label{rmk:morse_smooth}
	Because $\gamma$ is injective, we define 
	$f: A_\gamma \to [0, 1]$ by $f = \gamma^{-1} \circ s$, 
	where $s: A_\gamma\to \mathbb R^2$ is the function defined by \underline{Equation~\ref{eq:last_position}}. 
	Note this is well defined since $\gamma$ is an embedding. 
	This will serve as our Morse function. 
\end{remark}

In order to apply \underline{Lemma~\ref{farber4}}, we must construct an involution
whose fixed points are the same the critical points of $f$. 

\begin{proposition}\label{lem:walker33adapted}
	The critical points of $f$ are the collinear configurations. 
\end{proposition}

We will not give a proof for \underline{Proposition~\ref{lem:walker33adapted}} here, 
but it is adapted from Proposition 3.3 in \cite{walker} 
using \underline{Example~\ref{ex:add_edge}}. 
One crucial difference is the calculation of the index. 
For each circle of a critical radius, there are two types of transverse intersections:
those going ``in" and those going ``out." 
We assign this orientation based off the orientation of $\gamma$. 
Because $|\gamma(0)| = |\gamma(1)| = 1$, 
there exists a pairing of the intersections of each type.. 
This is why each $\mu_J$ is half the number of intersections
with the circle of radius $r_J$. 

Suppose $J \subseteq \{ 1, \ldots, k \}$ is ``long" and $\gamma$ has two intersections with the circle of radius $r_J$. 
From \cite{walker}, one of the associated critical points will have index $|J| - 1$. 
Because we will have the opposite orientation for the other associated critical point, 
its index will be $|K| = n - |J|$, where $K$ and $J$ are complimentary sets. 
Note how this lines up with our redefinition of $a_j$ and $b_j$, 
and that this corresponds with our original definition of $a_j$ and $b_j$ for straight lines. 

Lastly, we compute the homology groups by using \underline{Lemma~\ref{farber4}}. 
Here, our involution $\tau$ will be defined locally. 
Given $x \in \gamma$, 
there exists some angle $\theta$ between a ray from the origin through $x$ and the x-axis,
see \underline{Figure~\ref{fig:involution_angle}}. 
For each configuration $(\theta_1, \ldots, \theta_k)$ 
such that $s(\theta_1, \ldots, \theta_k) = x$, 
we say that
\begin{equation}\label{eq:moving_involution}
	\tau(\theta_1, \ldots, \theta_k) = (2\theta - \theta_1, \ldots, 2\theta - \theta_k)
\end{equation}
meaning that each configuration is reflected over the ray. 
Note that a collinear configuration whose end point is $x$
will have all angles $\theta$ or $\theta + \pi$, 
and as such will be fixed by the involution. 
Furthermore, the end points of each configuration will stay fixed, 
so $f$ is $\tau$ invariant. 

\begin{figure}[h]
	\centering
	\includegraphics[width=0.5\textwidth]{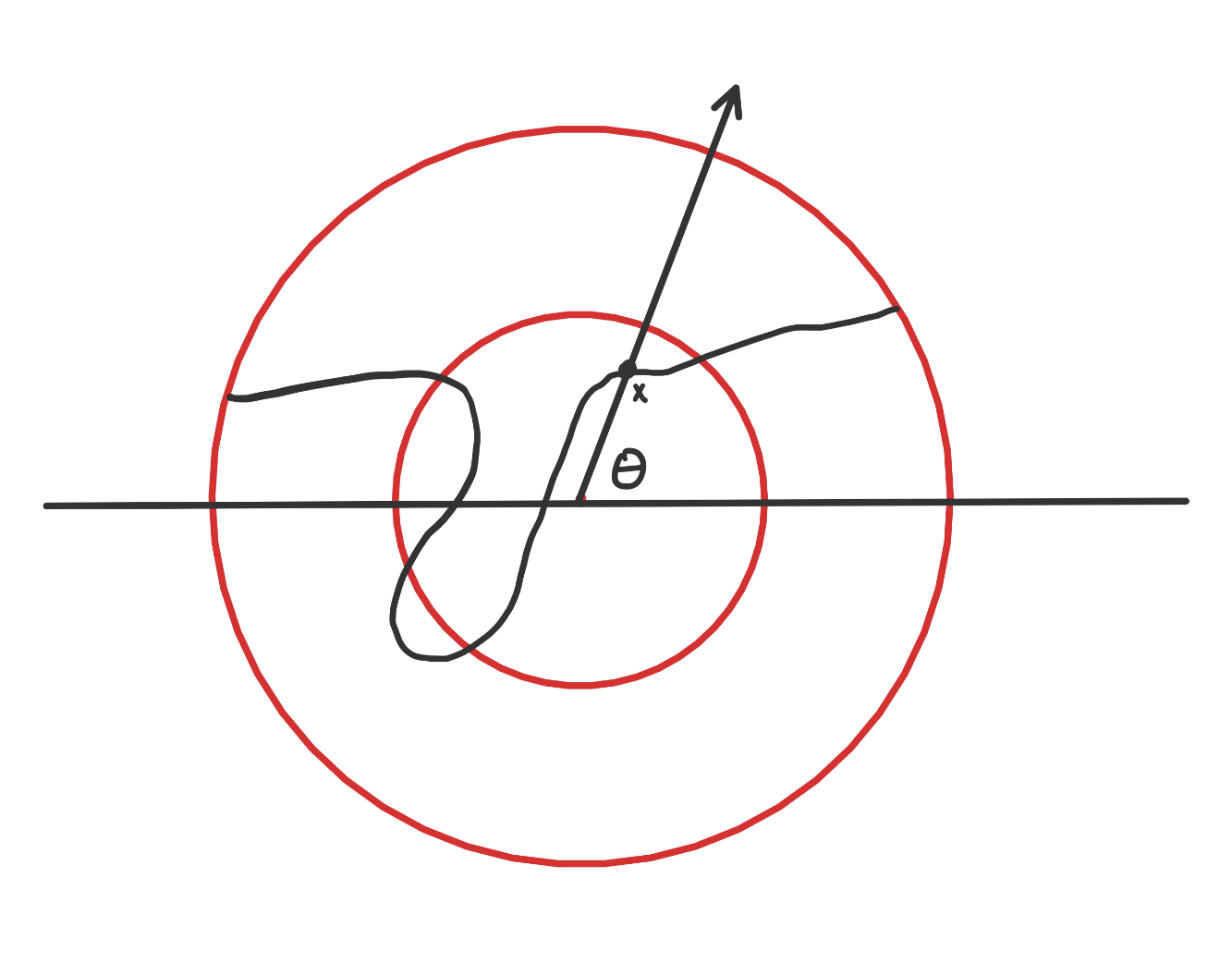}
	\caption{Angle of reflection $\theta$ for $x\in\gamma$}
	\label{fig:involution_angle}
\end{figure}

\begin{remark}
	Therefore, \underline{Lemma~\ref{farber4}} applies. 
	Thus for all $j \geq 0$, $H_j(A_\gamma)$ is free abelian with rank equal to the number of critical points
	with index $j$, i.e. $a_j + b_{j + 1}$. 
\end{remark}


\section*{Acknowledgements} 

Thank you to my thesis advisor, Professor Oleg Lazarev, for supporting me during this process. 
Thanks to Professors Alfred No{\"e}l and Stephen Jackson for pushing me to do an independent
study and thesis.


\end{document}